\documentclass[11pt,reqno]{amsart}

\usepackage{url}
\usepackage{graphicx}
\usepackage{mathrsfs}
\usepackage{color}
\usepackage{comment}
\usepackage{amssymb,amsthm,amsfonts}
\usepackage{enumitem}
\usepackage[margin=1.2in]{geometry}

\usepackage{xcolor}
\usepackage[colorlinks, linkcolor = blue, anchorcolor = blue, citecolor = blue,urlcolor= black]{hyperref}

\allowdisplaybreaks

\numberwithin{equation}{section}

\newtheorem{prop}{Proposition}[section]
\newtheorem{theo}[prop]{Theorem}
\newtheorem{lemm}[prop]{Lemma}
\newtheorem{coro}[prop]{Corollary}

\newtheorem*{claim*}{Claim}

\theoremstyle{definition}

\newtheorem{rema}[prop]{Remark}

\newcommand{\NN}{\mathbb{N}}

\newcommand{\RR}{\mathbb{R}}
\renewcommand{\SS}{\mathbb{S}}

\newcommand{\ZZ}{\mathbb{Z}}

\newcommand{\cB}{\mathcal B}

\newcommand{\cH}{\mathcal H}

\newcommand{\cQ}{\mathcal Q}

\def\fX{\mathfrak{X}}

\DeclareMathOperator{\vol}{vol}

\DeclareMathOperator{\genus}{genus}
\DeclareMathOperator{\ind}{index}

\newcommand{\bangle}[1]{\left\langle #1 \right\rangle}

\newcommand{\eps}{\varepsilon}
\renewcommand{\bar}{\overline}
\renewcommand{\hat}{\widehat}
\renewcommand{\tilde}{\widetilde}
\renewcommand{\leq}{\leqslant}
\renewcommand{\geq}{\geqslant}
\newcommand{\fh}[2]{\frac{#1}{#2}}

\renewcommand{\d}{\nabla}
\newcommand{\db}{\bar{\nabla}}
\newcommand{\la}{ \Delta}
\newcommand{\dt}{\de_t}
\newcommand{\di}{\de_i}
\renewcommand{\dj}{\de_j}
\renewcommand{\div}{\mathrm{div}}

\newcommand{\abs}[1]{\lvert#1\rvert}

\def\Xint#1{\mathchoice
	{\XXint\displaystyle\textstyle{#1}}%
	{\XXint\textstyle\scriptstyle{#1}}%
	{\XXint\scriptstyle\scriptscriptstyle{#1}}%
	{\XXint\scriptscriptstyle\scriptscriptstyle{#1}}%
	\!\int}
\def\XXint#1#2#3{{\setbox0=\hbox{$#1{#2#3}{\int}$ }
		\vcenter{\hbox{$#2#3$ }}\kern-.6\wd0}}

\def\dashint{\Xint-}

\newcommand{\ang}[1]{\langle #1 \rangle}

\newcommand{\weakto}{\rightharpoonup}

\newcommand{\de}{\partial}

\newcommand{\norm}[1]{\left\lVert#1\right\rVert}

\title[Topology of minimal submanifolds with finite total curvature]{Topology of complete minimal submanifolds in $\mathbb{R}^{n+m}$ with finite total curvature}

\author{Qi Ding}
\address{Shanghai Center for Mathematical Sciences, Fudan University, Shanghai 200438, China}
\email{dingqi@fudan.edu.cn}
\author{Lei Zhang}
\email{22110840013@m.fudan.edu.cn}
\date{}

\begin{document}

\begin{abstract}
In \cite{Chodosh17}, Chodosh, Ketover, and Maximo proved  finite diffeomorphism theorems for complete embedded minimal hypersurfaces of dimension $\leq 6$ with finite index and bounded volume growth ratio. 
In this paper, we adapt their method to study 
finite diffeomorphism types for complete immersed minimal submanifolds  of arbitrary codimension in Euclidean space with finite total curvature and Euclidean volume growth.
\end{abstract}

\maketitle

\setcounter{tocdepth}{1}
\tableofcontents

\section{Introduction}
For minimal surfaces in $\RR^3$, finite total curvature means that the Gaussian curvature integral is finite. Chern and Osserman \cite{Chern} proved that every minimal surface in $\RR^3$ with finite total curvature is conformally equivalent to a compact Riemann surface $\bar M$ punctured at a finite number of points, and the Gauss map on the surface can extend conformally to $\bar M$.
Collin \cite{collin1} proved that any properly embedded minimal
surface in $\RR^3$ with finite topology and more than one end, has finite total curvature. 
Colding and Minicozzi \cite{ColdingMincozzi1} removed the proper condition, where they proved that a complete embedded minimal surface $\Sigma$ with finite topology in $\RR^3$ must be proper. 
Meeks, Perez and Ros \cite{Meeks2} showed that the number of ends of $\Sigma$ is bounded by a constant depending on its genus.

Given an immersed minimal submanifold $M^{n}$  in $ \RR^{n+m}\,,$ $M$ is said to have \emph{finite total  curvature} if
 $$\int_{M} \abs{A}^{n}d\mu_M < \infty\,,$$
where $A$ denotes the second fundamental form of $M$ in $ \RR^{n+m}$, and $\mu_M$ denotes the volume element of $M$.

Anderson \cite{anderson1984compactification} gave a generalization of the
Chern-Osserman theorem \cite{Chern} on minimal surfaces of finite total curvature: a complete minimal submanifold $M^n$ with finite total curvature is diffeomorphic to a compact $C^{\infty} $ manifold $\bar{M}^n$ punctured at a finite number of points $\{p_i\}_{1}^{\ell} \in \bar{M}^n$ and the Gauss map $\gamma:M^n \to G_{n,m}$ extends to a $C^{n-2}$ map $\bar \gamma: \bar  M^n \to G_{n,m} $ of the compactification(where $G_{n,m}$ denotes the Grassmann manifold of $n$-planes in Euclidean $(n+m)$-space). 
In particular, $M$ has Euclidean volume growth with ratio bounded by a constant depending on $\ell$.
For complete minimal hypersurfaces in $\RR^{n+1}$ with $3 \leq n \leq 6\,,$ Tysk \cite{Tysk} proved that finite index and Euclidean volume growth imply finite total curvature.

Chodosh, Ketover, and  Maximo \cite[Theorem 1.1]{Chodosh17} proved that
for a fixed  closed Riemannian manifold $(M^n,g)(3 \leq  n \leq 7)\,,$ 
 there can be at most $ N = N(M,g,\Lambda,I)$ distinct diffeomorphism types in the set of
 embedded minimal hypersurfaces $\Sigma \subset (M,g)$ with $\ind(\Sigma) \leq  I$ and $\vol_g(\Sigma) \leq  \Lambda\,.$
 In particular, for $n=3\,$, there is $r_0 = r_0(M,g,\Lambda,I)$ so that
 any embedded minimal surface $\Sigma$ in $(M^3,g)$ with $\ind(\Sigma) \leq I$ and $\mathrm{area}_g(\Sigma) \leq \Lambda$ has
 $\genus(\Sigma) \leq  r_0$; for $4 \leq n \leq 7$, there is $N = N(n,I,\Lambda) \in  \NN$ so that there are at most
 $N$ mutually non-diffeomorphic complete embedded minimal hypersurfaces $\Sigma^{n-1} \subset  \RR^n$
 with $\ind(\Sigma) \leq I$ and $\vol(\Sigma\cap B_R(0)) \leq \Lambda R^{n-1}$ for all $R > 0$ \cite[Theorem 1.2]{Chodosh17}.

Buzano-Sharp \cite{sharp} proved both qualitative estimates on the total curvature and finitely-many diffeomorphism types of closed embedded minimal hypersurfaces with a priori bound on their index and area in closed Riemannian manifolds with dimension $\leq 7$. Antoine Song \cite{songantoine1} introduced a combinatorial argument and proved that for every closed embedded minimal hypersurface $\Sigma$ with area at most $A>0$  in a closed Riemannian manifold $(M^{n+1},g)$ with $3\leq n+1\leq 7$,  there is a constant $C_A>0$ depending only on $n$, $g$, and $A$ so that the sum of Betti number of $\Sigma$ is bounded above by $ C_A \big(1+\ind(\Sigma)\big)$. Edelen proved in \cite{edelen} that the space of smooth, closed, embedded minimal hypersurfaces $\Sigma$ in a closed Riemannian 8-manifold $(M^8,g)$ with a priori bounds $\cH^7(\Sigma) \leq \Lambda $ and $\ind(\Sigma) \leq  I$  divides into finitely-many diffeomorphism types, and this
finiteness continues to hold if one allows the metric $g$ to vary, or $\Sigma$ to be singular.

We get a finiteness result for minimal submanifolds under the conditions of uniformly bound total curvature and Euclidean volume growth. This can be seen as a quantitative generalization of Anderson's Theorem \cite{anderson1984compactification}.

\begin{theo} \label{theo.finite.top}
For fixed $n\,,m \in \ZZ^+\,,n \geq 3\,, m \geq1\,,$ and $ \Gamma \, ,\Lambda \in \RR \,,  \Gamma\,,\Lambda \geq 0\,$, there exists $N = N(n,m,\Gamma,\Lambda) \in \NN$ so that there are at most $N$ mutually non-diffeomorphic complete immersed minimal submanifolds $M^n$ in $\RR^{n+m}$ satisfying that  $\int_{M} \abs{A}^{n}d\mu_M  \leq \Gamma$ and 
$\vol_M ( B_R(0)) \leq  \Lambda R^n$ for any $R>0\,.$
\end{theo}


Our proof is inspired by the ideas in \cite{Chodosh17}, but we need further research in some situations. For instance, one point of concentration is a plane in  Proposition 7.1 of \cite{Chodosh17}, while in our situation it may be a non-flat minimal submanifold with finite total curvature. In Theorem \ref{prop.diffeo}, we can resolve it by an induction argument on the total curvature.

\subsection{Outline of the paper} 
\begin{itemize}
    \item In \S \ref{sec.pre}, we state several definitions
 and curvature estimates for minimal submanifolds which are needed in the following. 
 \item  In \S \ref{sec.geo}, we describe the geometry of ends of complete immersed minimal submanifolds in $\RR^{n+m}$ with finite total curvature, enlightened by \cite{schocen} and \cite{anderson1984compactification}.
This helps us to derive  curvature estimates away from finitely many points in Lemma \ref{lemm:curv.est}.
 \item In \S \ref{sec.top}, we prove a key topological result in Lemma \ref{lemm.Annular.composition} allowing us to control the topology of the “intermediate regions”, 
then combined the curvature estimates in \S \ref{sec.geo} we can prove Theorem \ref{theo.finite.top} by an induction argument on the total curvature.
\end{itemize}

\section{Notation and Preliminaries } \label{sec.pre}

\subsection{Definitions and basic notation}
For $n \geq 2, m\geq 1\,,$ given vectors $p,q \in \RR^{n+m}\,,$ let $\ang{p,q}$ denote the standard inner product between  vectors $p$ and $q\,.$
let $M$ be an $n$-dimensional complete smooth Riemannian manifold with boundary(possibly empty), and  $\iota: M \to \RR^{n+m}$ be the smooth isometric immersion($\iota|_{\de M}$ is also smooth). Here, the completeness of $M$ means that every geodesic from a point $p\in M \setminus \de M$ is defined until meeting some point in $\de M\,.$ At below, we define all kinds of notation on $M$ while they are defined only on $M\setminus \de M\,.$ We use the
notation $\fX(M)$ to denote the set of all smooth vector fields on $M\,.$ Let $\d$ and $\db$ be the Levi-Civita connections on $M$ and $\RR^{n+m}\,,$ i.e., $\d_X Y: = (\db_X Y)^T$ for $ X\,,Y \in \fX(M)$(we may identity $X$ and $d\iota(X)$ for $X\in \fX(M)$ since the differential calculations are local), where $(\cdots)^T$ denotes the projection onto the tangent bundle $TM$(see \cite{xin} for instance). In particular, $\d$ is induced from $\bar \d$ naturally. Let $\{e_i\}$ be a local orthonormal frame on $M\,,$ and $Y\in \fX(\RR^{n+m})\,,$ then $$\div_M Y := \sum_{i=1}^n\ang{\db_{e_i}Y, e_i}\,.$$ The second fundamental form $A$ of $M$ is defined by $$A(X,Y) := \db_XY - \d_XY = (\db_XY)^N$$ for vector fields $X,Y \in \fX(M)\,,$ where $(\cdots)^N$ denotes the projection onto the normal bundle $NM.$ Then we denote $\abs{A}^2$ as the square norm of $A\,$, i.e., $\abs{A}^2 = \sum_{i,j=1}^n \abs{A(e_i,e_j)}^2\,.$  Let $H$ denote the mean curvature vector of $M$ in $\RR^{n+m}$ defined by the trace of $A$, i.e., $H =\sum_{i=1}^n A(e_i , e_i )\,$, which is a normal vector field on $M$.  If $H \equiv 0$ on $M\,,$ then $M$ is  a complete  immersed minimal submanifold in $\RR^{n+m}$ with boundary. 

  Given $ \lambda >0\,,$ and $ q \in \RR^{n+m}\,,$ after rescaling in $\RR^{n+m}\,,$  we get a new immersion $\hat \iota: M \to \RR^{n+m}\,, \  \hat \iota(x) = \lambda(\iota(x)-q)$  for all $x\in M\,.$ We denote $\hat M : =\lambda(\iota(M)-q)$ as the new immersed submanifold.
  Since we can pull back the Riemannian metric from $\RR^{n+m}$ to $\hat  M\,.$ We see that $\hat  M$ is also a complete immersed minimal submanifold in $\RR^{n+m}$ with boundary.  At below, given $q \in  \RR^{n+m}\,,$ we denote $\abs{A_{M}}^2(q)$ as the square norm of the second fundamental form of some point in $M$ whose image is $q \,,$ and this point is concrete from the context.

 For points $p,q \in \RR^{n+m}\,,$ let $|p-q|$ be the Euclidean distance between points $p$ and $q\,.$ Given $r>0\,,$ we denote $B_r(q) = \{p \in \RR^{n+m}| |p-q| <r\}\,,$ and $\bar{B_r(q)}$ is the closure of $B_r(p)$ in $\RR^{n+m}\,,$ i.e., $\bar{B_r(q)} = \{p \in \RR^{n+m}| |p-q| \leq r\}\,.$ For subset $U\subset \RR^{n+m}$ and $r>0\, ,$  let $$B_r(U) := \bigcup_{p\in U} B_{r}(p)\, . $$  For $x,y \in M\,,$ let $d_M(x,y)$ be the (Riemannian) distance between $x$ and $y$ on $M.$ Then  $B^M_r(x) = \{y \in M| d_M(x,y) <r\}\,,$ and the closure of $B^M_r(x)$ in $M$ is  $\bar{B^M_r(x)} = \{y \in M| d_M(x,y) \leq r\}\,.$ Abusing notation slightly, we denote $M\cap B_R(p)$ as $M\cap \iota^{-1}(B_R(p))$ and denote $\vol(B_R(p))$ as $\vol(M\cap B_R(p))\,.$ 

Given a set $G$ of finite elements, we denote $\abs G$ as the number of elements in the set $G\,.$ Given subsets $U\, , V \subset \RR^{n+m},$ we denote $d_{\cH}(U,V)$ as the Hausdorff distance between $U$ and $V, $
i.e., $$d_{\cH}(U,V) = \inf \{\varepsilon >0| V \subset B_\eps(U) \text{ and } U\subset B_{\eps}(V)\}\, . $$

 In \cite{Chodosh17}, they defined \emph{Smooth blow-up sets} for embedded minimal hypersurfaces. Here, we define a similar concept for immersed minimal submanifolds in $\RR^{n+m}\,.$  Suppose that $M_{j}$ is a sequence of complete immersed minimal submanifolds with boundary(possibly empty) in $\RR^{n+m}$. A sequence of  subsets $\cB_{j} \subset M_{j}$ with $|\cB_j|< \infty$  is said to be \emph{a sequence of smooth blow-up sets} if:
	\begin{enumerate}[itemsep=5pt, topsep=5pt]
		\item The set $\iota_j(\cB_{j})$ remains a finite distance from the base point $0 \in \RR^{n+m}\,$, i.e.,
		$$
		\limsup_{j\to\infty}\max_{p\in \cB_{j}} |\iota_j(p)| < \infty\, .
		$$
		\item If we set $\lambda_{j}(p) : = |A_{M_{j}}|(p)$ for $p \in \cB_{j}\,$, then the curvature of $M_{j}$ blows up at each point in $\cB_{j}\,$, i.e.,
		$$
		\liminf_{j\to\infty} \min_{p\in\cB_{j}} \lambda_{j}(p) = \infty\, .
		$$
		\item If we choose a sequence of points $p_{j}\in \cB_{j}$, then after passing to a subsequence, the rescaled submanifold $\tilde M_{j}:=\lambda_{j}(p_{j})(\iota_j(M_{j}) - \iota_j(p_{j}))$ converges locally smoothly to a complete, non-flat, immersed minimal submanifold $\tilde M_{\infty}\subset\RR^{n+m}$ without boundary, satisfying
		$$
		|A_{\tilde M_{\infty}}|(x) \leq |A_{\tilde M_{\infty}}|(0)\footnote{$|A_{\tilde M_\infty}(0)|$ is the value of some point in $\tilde M_\infty$ with image $0\in \RR^{n+m}\,.$},
		$$
		for all $x\in \tilde M_{\infty}\,$. 
		\item The blow-up points do not appear in the blow-up limit of the other points, i.e.,
		\[
		\liminf_{j\to\infty}\min_{\substack{p,q \in \cB_{j}\\ p\not=q}} \lambda_{j}(p) |\iota_j(p)-\iota_j(q)| = \infty\, .
		\]
	\end{enumerate}

\subsection{Curvature estimates.}
Choi and Schoen \cite{Choi1} proved curvature estimates under small total curvature condition for minimal surfaces. Furthermore, Anderson \cite{anderson1984compactification} proved curvature estimates under small total curvature condition for $n$-dimensional minimal submanifolds in $\RR^{n+m}\,.$ 
Here, we state one slightly different from Anderson's result \cite{anderson1984compactification} as follows.
\begin{lemm}[Curvature estimates in the extrinsic distance] 
\label{curvature.estimate1}
   For fixed $n\,,  m \in \ZZ^+\,, n\geq 2\,,m\geq 1\,,$ there exists  $C_1\, ,S_{n,m}> 0$ depending on $n\,,m$ such that if $M^n(\iota :M^n \to \RR^{n+m})$ is a complete  properly immersed minimal submanifold with nonempty boundary and the total curvature $\int_M \abs{A}^n d\mu_M < S_{n,m}\,,$ then $\abs{A}(x)d(\iota(x),\iota(\de M) )< C_1$ for all $x \in M\,.$
\end{lemm}
\begin{proof}
Let us argue by contradiction. If the lemma is false, then there must have a sequence of complete properly immersed minimal submanifolds $M_j$ satisfying that the total curvature 
$$\alpha_j := \int_{M_j} \abs{A_{M_j}}^n d\mu_{M_j}\to 0\,,$$ but 
    $$\beta_j := \sup_{x\in M_j}\abs{A_{M_j}}(x)d(\iota_j(x),\iota_j(\de M_j)) \to \infty\,.$$
    Then the standard point picking argument(see \cite{Chodosh17} Lemma 2.2) by passing to a subsequence allows us to find $\tilde{q}_j \in M_j$ so that for $\lambda_j := \abs{A_{M_j}}(\tilde{q}_j)\,,$  the rescaled minimal submanifold 
    $$\tilde{M}_j := \lambda_j(\iota_j(M_j) -\iota_j(\tilde{q}_j))$$
    converges locally smoothly in $\RR^{n+m}$ to a complete immersed minimal submanifold $\tilde M_\infty\,.$  Moreover, $\tilde M_\infty$ has no boundary and the total curvature of $\tilde M_\infty$ equals  0. While $\abs{A_{\tilde M_\infty}}(0) = 1,$
    which derives a contradiction.
    
    For the convenience  of readers, we recall the point picking argument used above to construct $\tilde M_{\infty}$. Let $\iota_j: M_j \to \RR^{n+m}$ denote the immersion map. Choose $\tilde p_{j} \in M_{j}$ so that 
$$
|A_{M_{j}}|(\tilde p_{j}) d(\iota_j(\tilde p_{j}),\iota_j(\partial M_{j})) > \fh{1}{2}\beta_{j} \to \infty
$$
and set $r_{j} = |A_{M_{j}}|(\tilde p_{j})^{-\fh{1}{2}} d(\iota_j(\tilde p_{j}),\iota_j(\partial M_{j}))^{\fh{1}{2}}\,$. Then, we choose $\widetilde q_{j} \in M_{j}\cap \iota_j^{-1}(B_{r_{j}}(\iota_j(\widetilde p_{j})))$ so that
\begin{equation} \label{eq:max.choose1}
|A_{M_{j}}|(\widetilde q_{j}) d(\iota_j(\widetilde q_{j}),\partial B_{r_{j}}(\iota_j(\widetilde p_{j}))) = \max_{\iota_j(x) \in B_{r_{j}}(\iota_j(\widetilde p_{j}))}|A_{M_{j}}|(x) d(\iota_j(x),\partial B_{r_{j}}(\widetilde p_{j}))\, .
\end{equation}
Note that the right hand side is at least $\left(|A_{M_{j}}|(\tilde p_{j})d(\iota_j(\tilde p_{j}),\iota_j(\partial M_{j}))\right)^{\frac 12}$  which is tending to infinity. Let $R_{j} = d(\iota_j(\widetilde q_{j}),\partial B_{r_{j}}(\iota_j(\widetilde p_{j})))\, $. Because $d(y,\partial B_{R_{j}}(\iota_j(\widetilde q_{j}))) \leq d(y,\partial B_{r_{j}}(\iota_j(\widetilde p_{j})))$ for any $y \in B_{R_{j}}(\widetilde q_{j})\, $, we find that 
\begin{equation} \label{eq:max.choose2}
|A_{M_{j}}|(\widetilde q_{j}) d(\iota_j(\tilde q_{j}),\partial B_{R_{j}}(\iota_j(\tilde q_{j}))) = \max_{\iota_j(x) \in B_{R_{j}}(\iota_j(\widetilde q_{j}))}|A_{M_{j}}|(x) d(\iota_j(x),\partial B_{R_{j}}(\iota_j(\widetilde q_{j})))\, .    
\end{equation}
Note that $|A_{M_{j}}|(\widetilde q_{j})R_{j} \geq |A_{M_{j}}|(\widetilde p_{j})r_{j} \to \infty\, $. 

As above, we set $\lambda_{j} = |A_{M_{j}}|(\tilde q_{j})\,$. Then, the rescaled submanifold
$$
\tilde M_{j} =\lambda_{j}(\iota_j(M_{j}) - \iota_j(\tilde q_{j}))
$$
with immersion map $\tilde \iota_j $ satisfies
$$
|A_{\tilde M_{j}}|(x) d(\tilde\iota_j(x)\,,\partial B_{\lambda_{j}R_{j}}(0)) \leq \lambda_{j}R_{j}\,,
$$
when $\tilde \iota_j(x) \in B_{\lambda_{j}R_{j}}(0)\,.$ If $x\in\tilde M_{j}$ and $\tilde \iota_j(x)$ lies in a given compact set of $\RR^{n+m}\,$, then
$$
|A_{\tilde M_{j}}|(x) \leq \frac{\lambda_{j} R_{j}}{\lambda_{j}R_{j} - |\tilde \iota_j(x)|} \to 1 = |A_{\tilde M_{j}}|(0) 
$$
as $j\to\infty\,$. Then $d(0,\tilde \iota_j(\partial \tilde M_j)) \to \infty$ due to $\lambda_j R_j \to \infty\,$. After passing to a subsequence, we can take a smooth limit of $\lambda_{j}(\iota_j(M_{j}) - \iota_j(\widetilde q_{j}))$ and find a complete, non-flat, immersed minimal submanifold $\tilde M_{\infty}$ in $\RR^{n+m}$ without boundary.
\end{proof}

At below, for fixed $n\,,m \in \ZZ^+\,, n\geq2\,,m\geq1\,,$ we fix $K_0 =\fh{1}{2} S_{n,m}$ where $S_{n,m}$ is a fixed  positive number satisfying Lemma \ref{curvature.estimate1}.
 \begin{rema} \label{rema.cur.est}
We also have  curvature estimates in the intrinsic distance. But we do not need to assume the immersion is proper. While in the extrinsic case, we assume that the  immersion is  proper to ensure the  maximum can be achieved in \eqref{eq:max.choose1} and \eqref{eq:max.choose2}.
\end{rema}
 \begin{lemm}[Curvature estimates in the intrinsic distance] 
\label{curvature.estimate2}
    For fixed $\delta> 0\,,  n\,,m\in \ZZ^+\,,n\geq 2\,,m \geq 1\,, $ there exists  $\varepsilon_2> 0$  such that if $M^n(\iota :M^n \to \RR^{n+m})$ is a complete connected immersed minimal submanifold in $\RR^{n+m}$  with nonempty boundary and the  total curvature $\int_M \abs{A}^n d\mu_{M}< \eps_2\,,$ then $\abs{A}(x)d^M(x,\de M )< \delta$ for all $x \in M\,.$ 
\end{lemm}
\begin{proof}
We argue by contradiction. If the lemma is false, there must have a sequence of complete immersed minimal submanifolds $M_j$ with the total curvature $$\alpha_j := \int_{M_j} \abs{A_{M_j}}^n  d\mu_{M_j}\to 0\,,$$ but 
$$\beta_j := \sup_{x\in M_j}\abs{A_{M_j}}(x)d^{M_j}(x,\de M_j) \geq 2C_2>0\,.$$
Then we can find $\tilde{q}_j \in M_j$ such that $\abs{A_{M_j}}(\tilde{q}_j)d^{M_j}(\tilde{q}_j,\de M_j) > C_2\,.$ After rescaling, we can assume $\abs{A_{M_j}}(\tilde{q}_j) = 1, \ d^{M_j}(\tilde{q}_j,\de M_j)> C_2 $ and $\iota_j(\tilde{q}_j) = 0\,.$ If $j$ is sufficiently large, then $|A_{M_j}|(x)$ is uniformly bounded for any $x \in B^{M_j}_{\fh{1}{2}C_2}(\tilde{q}_j)$ by  curvature estimates stated in Remark \ref{rema.cur.est} similar to Lemma \ref{curvature.estimate1}. We denote $\hat\iota_j$ as $\iota_j$ restricted on $M_j\cap  B^{M_j}_{\theta C_2}(\tilde{q}_j)$ for some $\theta $ small. By taking $\theta$ sufficiently small, we can assume $\hat \iota_j$ is an embedding and the image of $\hat \iota_j$ in $\RR^{n+m}$ is the graph of some function $u_j\,.$  After passing to a subsequence, we can assume $\hat \iota_j$  converges locally smoothly to a  minimal embedding  $\hat \iota_\infty: \hat M_\infty \to \RR^{n+m}$ whose image  in $\RR^{n+m}$ is also the graph of some function $u_\infty $ and $\iota_j(\tilde q_j) =  \iota_\infty( \tilde q_\infty) = 0 \in \RR^{n+m}\,.$ Hence $\abs{A_{\hat M_\infty}}(\tilde q_\infty) = 1 $ but $$\int_{\hat M_\infty} \abs{A_{\hat M_\infty}}^n d\mu_{M_\infty}= 0\,,$$ which is a contradiction.

\end{proof}


\section{Geometry of Minimal Submanifolds with Finite Total Curvature} \label{sec.geo}

 For fixed $n,m \in \ZZ^+, n\geq 3, m \geq 1\,,$ a complete minimal immersion $\iota:M^n \to \RR^{n+m}$ is said to be \emph{regular at infinity} if there is a compact subset $K \subset  M$ such that $M \setminus K$ consists of $r$
 components $M_1,\cdots ,M_r$
 satisfying that each $\iota(M_i)$
 is the graph of the vector-valued function $Y_i= (Y_{i}^1\,,\cdots, Y_{i}^m)$ defined over the exterior of a bounded region in some $n$-plane $\Pi_i$.
 Moreover, if $x^1\,,\cdots ,x^n$
 are coordinates in $\Pi_i$
 ,  the function $Y_{i}^\ell$ has the
 following asymptotic behavior for $\abs{x}$ large,
 $$Y_{i}^\ell = b_i^\ell + a_i^\ell\abs{x}^{2-n} + \sum_{j=1}^nc_{ij}^\ell x^j\abs{x}^{-n} +O(\abs{x}^{-n})\,, 1\leq \ell \leq m\,, 1\leq i \leq r\,.$$
If $m=1\,,$ the above definition is consistent with the definition of \emph{regular at infinity} as Schoen in \cite{schocen}. From the definition of regular at infinity, $M$ has finite ends and each end is an embedded minimal submanifold in $\RR^{n+m}\,.$ Moreover,  $\vol(M\cap  B_R(0)) \leq \Lambda R^n$ for any $R>0$ with  $\Lambda>0$ equaling the number of ends of $M$ by the monotonicity formula(see \cite{GMT} for more details about monotonicity formula).

Since up to a rotation, every end can be described as a minimal graph over the exterior of a bounded region in $\RR^n\times \{0^m\} \subset \RR^{n+m}\,.$  We have a parametrization for a minimal graph, i.e.,
\begin{gather}
   \begin{aligned} \label{func.minimal.graph}
      \Psi:  \SS^{n-1} &\times  (a,b)  \to \RR^{n+m}\, , 0\leq a < b \leq \infty\,,\\
      &( x,t) \mapsto (e^tx,e^tF(x,t))\,.  
   \end{aligned} 
   \end{gather}
   We have $e^tx \in \RR^{n}\,,$ $\SS^{n-1} \subset \RR^n \times \{0^m\} $ and  $F$ is a smooth vector-valued function defined on a domain of $ \SS^{n-1} \times \RR$ with $F(t,x) \in \RR^m\,.$
Then we compute the minimal surface system(see Appendix \ref{Appendix.A} for more details about calculations) and get
   \begin{gather}
    F_{tt} + nF_t +(n-1)F + \la_{\SS^{n-1}}F +\cQ(F) = 0 \,.\label{eq.minimal.graph}
\end{gather}
The linearized operator of the  equation \eqref{eq.minimal.graph} is 
\begin{gather}
L(F)=F_{tt} + nF_t +(n-1)F + \la_{\SS^{n-1}}F \,.
\end{gather}
Our analysis of solutions of \eqref{eq.minimal.graph} is based on the asymptotic behavior
of elements in the kernel of $L$. Such an element in the kernel can be decomposed as
the sum of terms of  $u(t)\Phi(x) $  with $\Phi (x) \in \RR^{m}\,,$ where $\Phi $  is
a vector-valued eigenfunction of the Laplace operator on $\SS^{n-1}\,.$
The $k^{th}$ eigenvalue of $\la_{\SS^{n-1}}$ on $\SS^{n-1}$ is $-k(k+n-2)$ ($k \in \NN $). So
$(x,t)\mapsto u(t)\Phi(x)$ is in the kernel of $L$ if $u$ satisfies the
following ordinary differential equation for some $k$ :
\begin{gather} \label{asy.ode}
u_{tt}+nu_t+(n-1-k(k+n-2) )u = 0\,.
\end{gather}
Then we solve the equation \eqref{asy.ode}, and get $u(t) = C_{k,\pm }e^{\lambda_{k,\pm}t}$ with
 $\lambda_{k,\pm} =-\fh n 2 \pm (\fh n 2 +k-1)\,.$ So $\lambda_{k,+} = k-1\,, \lambda_{k,-} = -n+1-k\,.$

 Recall a classical definition of weighted norm for vector-valued functions on
$\SS^{n-1} \times  \RR^+$. If $Y$ is a continuous vector-valued function from $\SS^{n-1} \times
\RR^+$  to $\RR^m$ and $\beta\in \RR\,,$ we define its weighted norm
\begin{gather*}
    \|Y\|_{s,\beta}:=\sup \{e^{\beta t}|Y(x,t)|_s| (x,t)\in \SS^{n-1} \times \RR^+\}
\,,s \in \NN\,. \\
|Y(x,t)|_s : = |\d^sY(x,t)|_0+|Y(x,t)|_{s-1} \,, s \in \NN\,, s \geq1\,.\\
|\d^sY(x,t)|_0 : =\norm{\d^s Y(x,t)}_{C^0(\SS^{n-1} \times  \RR^+)}\,, s\in \NN \,.
\end{gather*}
 If $\|Y\|_\beta: = \|Y\|_{0,\beta} < \infty\,$, we will also write
$Y=O(e^{-\beta t})\,$.


\begin{prop}\label{prop.improv}
Let $Y$ be a solution of \eqref{eq.minimal.graph} on $\SS^{n-1} \times \RR^+ $ satisfying
that $|\d Y|_0<\infty\,,$ $\|Y\|_\beta<\infty$ with $\beta>0$
and $-3\beta \neq \lambda_{k,\pm}$ for all $k \geq 0\,$. Then $Y$ can be written
$Y=X+R\,,$ where $\|X\|_\beta <\infty$ satisfying that  $L(X)=0$ and
$\|R\|_{3\beta}<\infty\,$.
\end{prop}
\begin{proof}
    The proof is based on the spectral decomposition of vector-valued functions on
$\SS^{n-1}$. 

Since  $|\d Y|_0 <\infty$ and Equation \eqref{eq.minimal.graph} is
uniformly elliptic, the classical elliptic estimates give upper bounds on the
derivatives of $Y$: more precisely, for any $\ell>0\,$, there is a constant $C_\ell'$ independent of $s$
such that for any $s>1\,,$
$$
\|\nabla^\ell Y \|_{C^0(\SS^{n-1}\times [s,s+1])}\leq C_\ell'\|Y\|_{C^0(\SS^{n-1}\times [s-1,s+2])}\,.
$$
This implies that for any $\ell>0\,$, $\|\d^\ell Y \|_\beta<\infty\,$. Since the term
$\cQ(Y)$ in \eqref{eq.minimal.graph} gathers all the nonlinear terms consisting of $Y\,, \d Y\, , \d^2Y$  at least cubic(see Appendix \ref{Appendix.A}), we have
$\|\cQ(Y)\|_{3\beta}<\infty$ and $\|\d^\ell \cQ(Y)\|_{3\beta}<\infty\,$. 

In the preceding section, we have described the spectrum of the Laplace operator
on the
sphere. So let us denote $\lambda_k:=k(k+n-2)$ and $\Phi_{k,\alpha}$ the 
orthonormal basis of the eigenspace of $\la_{\SS^{n-1}}$ associated to $-\lambda_k\,,$ that is $$\dashint_{\SS^{n-1}}\ang{\Phi_{k_1,\alpha_1}\,,\Phi_{k_2,\alpha_2}} d\mu_{\SS^{n-1}} = \delta_{k_1k_2}\delta_{\alpha_1\alpha_2}\,.$$ The
dimension of the eigenspace associated to $-\lambda_k$ is bounded by $c_1
(k^{n-1}+1)m$ with $c_1$ only depending on $n\,.$  Moreover, for $k \geq 1\,,$ we have the following estimates for the
$L^\infty$ norm of the eigenfunctions (see \cite{Sogge1}):
$$
\|\Phi_{k,\alpha}\|_\infty\leq c_2\lambda_k^{\frac{n-2}4}
, c_2>1 \text{ only depending on } n\,.$$

Now let us define
\begin{align*}
g_{k,\alpha}(t)&=\dashint_{\SS^{n-1}} \ang{Y(x,t), \Phi_{k,\alpha}(x)} d\mu_{\SS^{n-1}}
 \,,\\
f_{k,\alpha}(t)&=-\dashint_{\SS^{n-1}} \ang{\cQ(Y)(x,t),
\Phi_{k,\alpha}(x)} d\mu_{\SS^{n-1}}\,.
\end{align*}

 Hence $g_{k,\alpha}$ and $f_{k,\alpha}$ are smooth functions on $\RR^+\,,$
and from \eqref{eq.minimal.graph}, they satisfy
\begin{gather}
g_{k,\alpha}''-(\lambda_{k,+}+\lambda_{k,-})g_{k,\alpha}'+
(\lambda_{k,+}\times\lambda_{k,-})g_{k,\alpha}= f_{k,\alpha}\,.
\end{gather}

Using $\la_{\SS^{n-1}} \Phi_{k,\alpha}=-\lambda_k\Phi_{k,\alpha}\,$, and integration by parts, for $k\geq 1\,,$ we get the
following estimates for $a,b\in\ZZ^+$:
\begin{align*}
|g_{k,\alpha}(s)|&\leq \frac{\sup_{t=s}|\nabla^{2a}Y(x,t)|}
{(1+\lambda_k)^a}\, ,\\
|f_{k,\alpha}(s)|&\leq \frac{\sup_{t=s}|\nabla^{2a}\cQ(Y)(x,t)|}
{(1+\lambda_k)^a}\,.
\end{align*}
Thus we get 
\begin{align*}
\|g_{k,\alpha}\|_\beta&\leq \frac{\|\nabla^{2a}Y\|_\beta}
{(1+\lambda_k)^a} \,,\\
\|f_{k,\alpha}\|_{3\beta}&\leq \frac{\|\d^{2a}\cQ(Y)\|_{3\beta}}
{(1+\lambda_k)^a}\,.
\end{align*}

For $k=0\,,$ 
$$ \|f_{k,\alpha}\|_{3\beta}\leq \| \cQ(Y) \|_{3\beta}\,.$$

From the standard ordinary differential equation theory(see Lemma 10 in Appendix A of \cite{Mazet1}), we can write
$$
g_{k,\alpha}(t)=a_{k,\alpha}e^{t\lambda_{k,+}} +
b_{k,\alpha}e^{t\lambda_{k,-}}+ r_{k,\alpha}(t)
$$
with some estimates on the different terms. First we notice that
$|\lambda_{k,+}-\lambda_{k,-}| = |2k+n -2|$ and $|3\beta+\lambda_{k,\pm}|$ are uniformly
bounded from below from $0$ and
$\frac{(2+|\lambda_{k,+}|^2+|\lambda_{k,-}|^2)^{1/2}}
{|\lambda_{k,+}-\lambda_{k,-}|}$ is uniformly bounded. Hence, for $k\geq 1\,,$ there
is a uniform constant $c_3$ independent of $k$ such that
\begin{align*}
\max(|a_{k,\alpha}|,|b_{k,\alpha}|)&\leq c_3 (
\|g_{k,\alpha}\|_\beta+ \|g'_{k,\alpha}\|_\beta+
\|f_{k,\alpha}\|_{3\beta})\\
&\leq c_3 \frac{\|\d^{2a}Y\|_\beta+ \|\d^{2a+1}Y\|_\beta+
\|\d^{2a}Q(Y)\|_{3\beta}}{(1+\lambda_k)^a}
\end{align*}
and  
\begin{align*}
\|r_{k,\alpha}\|_{3\beta}&\leq c_3 \|f_{k,\alpha}\|_{3\beta}\\
&\leq c_3 \|\cQ(Y)\|_{3\beta}\,.
\end{align*}
For $k=0\,,$ $$\max(|a_{k,\alpha}|,|b_{k,\alpha}|) \leq C'\, , \|r_{k,\alpha}\|_{3\beta} \leq C'\,.$$ 
If $\lambda_{k,+} > -\beta$, $\norm{e^{t\lambda_{k,+}}}_{\beta} = \infty$, so $a_{k,\alpha} = 0\,.$ Also, if $\lambda_{k,-} > - \beta\,$,  $b_{k,\alpha} = 0\,.$ If $\lambda_{k,\pm}\leq -3\beta\,$,
$ \norm{e^{t\lambda_{k,\pm}}}_{3\beta} = 1$.

Finally we have the following equality
\begin{gather} \label{eq.pro.sum}
\begin{aligned}
Y(x,t)&=\sum_{-3\beta\leq \lambda_{k,+}\leq -\beta} a_{k,\alpha}
e^{t\lambda_{k,+}} \Phi_{k,\alpha}(x)
+\sum_{-3\beta\leq \lambda_{k,-} \leq -\beta} b_{k,\alpha} e^{t\lambda_{k,-}}
\Phi_{k,\alpha}(x)\\
&\quad+\sum_{\lambda_{k,+} < -3\beta} a_{k,\alpha}
e^{t\lambda_{k,+}} \Phi_{k,\alpha}(x)
+\sum_{\lambda_{k,-}< -3\beta}b_{k,\alpha} e^{t\lambda_{k,-}}
\Phi_{k,\alpha}(x)\\
&\quad+\sum_{k=0}^\infty r_{k,\alpha}(t) \Phi_{k,\alpha}(x)\, .
\end{aligned}
\end{gather}
First we notice that the first two sums of \eqref{eq.pro.sum} are finite and are elements of the kernel of
$L$, this is the expected function $X$. In fact, we claim that the other sums
converge and have finite $3\beta $-norms. Let $A(x,t)$ be the sum of the term with
$\lambda_{k,-}< -3\beta $ . In
the following computation, we use the expressions of $\lambda_k\,,$
their multiplicities and the $L^\infty$ estimates on $\Phi_{k,\alpha}\,.$
\begin{align*}
\|A\|_{3\beta }&\leq c_2\sum_{\lambda_{k,-}< -3\beta,k \geq 1} |b_{k,\alpha}|
 \lambda_k^{\frac{n-2}4} + \sum_{\lambda_{0,-}<-3 \beta}|b_{0,\alpha}|\\
&\leq c_2c_3\sum_{k\geq1,\alpha} \frac{\|\d^{2a}Y\|_\beta+
\|\d^{2a+1}Y\|_\beta+
\|\d^{2a}\cQ(Y)\|_{3\beta}}{(1+\lambda_k)^a}
\lambda_k^{\frac{n-2}4} + c_1C'm\\
&\leq 2c_1c_2c_3m(\|\nabla^{2a}Y\|_\beta+ \|\d^{2a+1}Y\|_\beta+
\|\d^{2a}\cQ(Y)\|_{3\beta})\sum_{k=1}^\infty \fh{(1+k^{\frac{3n}2})
}{(1+k^2)^a }+c_1C'm\\
&<\infty
\end{align*}
if $a$ is chosen such that $2a-\frac{3n}2\geq 2\,.$ We can prove the other two sums in the claim by the same method and we omit the details.
\end{proof}

  Note that Schoen \cite[Proposition 1]{schocen} have shown that a complete immersed minimal  surface $M^2 \subset \RR^3$ is regular at infinity
 if and only if $M$ has finite total curvature and each end of $M$ is embedded. Furthermore,
    Schoen \cite[Proposition 3]{schocen} showed that if $ n \geq 3$, and $M^n \subset \RR^{n+1}$ is a minimal immersion with
 the property that $ M \setminus K$, for some compact subset $K \subset M$, is an union of $M_1,\cdots,M_r$
 where each image of $M_i$ in $\RR^{n+1}$ is a graph of bounded slope over the exterior of a bounded region in a
 hyperplane $\Pi_i\,,$ then $M$ is regular at infinity.

Moreover, Anderson \cite[Theorem 3.2]{anderson1984compactification} showed that item (1) can imply item (3) in the following Theorem \ref{eqival1} and we resolve it  using a different method. We
 refer  readers to the original papers for more details.

\begin{theo} \label{eqival1}
  For fixed $\ n,m \in \ZZ^+,n\geq 3,m\geq 1\,,$ if $\iota : M^n \to \RR^{n+m} $ is a complete connected  minimal immersion in $\RR^{n+m}$, then  the following statements are equivalent:
    \begin{enumerate}[itemsep=5pt, topsep=5pt]
        \item The total curvature of $M$ is finite.
        \item  $M$ is of finite ends and each end $E$ of $M$ has a tangent cone at infinity as an $n$-plane with multiplicity one, i.e., $|r_i^{-1}\iota(E)| \rightharpoonup |\psi(\RR^n\times \{0^m\} )|$  in the sense of varifolds in $\RR^{n+m} \setminus B_1(0)$ where $ r_i\uparrow \infty$ and $\psi \in SO(n+m). $  
        \item $M$ is regular at infinity. 
    \end{enumerate}
\end{theo}
\begin{proof}
Fix a point $p \in M\,,$ up to a translation, we can assume $\iota(p) = 0\in \RR^{n+m}\,.$

    $(1)\Rightarrow(2):$ This has been proved in \cite[Theorem 3.1]{anderson1984compactification}. We have organized his proof as follows. We firstly prove that $\iota$ is a  proper immersion. Since the restriction of coordinate functions on $M$ is harmonic  on $M\,,$  $M$ is not a closed manifold. Let $f(x)$ denote the function $|\iota(x)|$ and $X = \iota(x)$ as the position vector. By Lemma \ref{curvature.estimate2}, we can choose $R_0$ large such that for any $q \in M$ if  $d^M(q,p) \geq R_0\,,$ then $\abs{A}(q)d^M(q,p)< \fh{1}{4}\,.$ Let $t_0 := d^M(q,p)$ and $\gamma(t)$ be the minimizing normal geodesic in $M$ between $p$ and $q$ with $\gamma(0) = p$ and $\gamma(t_0) =q\,.$ Let $V = \gamma'(t)\,,$ and if $t \geq R_0\,,$ then we have   
    $$V\ang{V,X} = \ang{A(V,V),X} + 1\geq 1-\abs{A}|X| \geq \fh{3}{4}\,.$$   
    At the point $q\,,$ 
\begin{gather}
    \begin{aligned}
   f(q)  &\geq \ang{X,V}(t_0)\\
  & = \ang{X,V}(R_0) +\int_{R_0}^{t_0} V\ang{V,X}(t) dt\\
  & \geq \ang{X,V}(R_0) + \fh{3}{4}(t_0 - R_0)    . 
   \end{aligned}
    \end{gather}                     So $\iota$ is a proper immersion. 
    We compute $\abs{\d f}$ at $\gamma (t_0) = q\,,$
\begin{gather}
    \begin{aligned}
        \abs{\d f}(q) &\geq \fh{\ang{X,V}(t_0)}{f(q)}\\ &\geq \fh{\ang{X,V}(t_0)}{t_0}\\
        &\geq \fh{\ang{X,V}(R_0)-\fh{3}{4}R_0}{t_0} + \fh{3}{4}\\
        &\geq -\fh{2R_0}{t_0} + \fh34 .
    \end{aligned}
\end{gather}  
If $t_0 \geq 8 R_0\,,$ then $\fh 12\leq \abs{\d f}(t_0) \leq 1\,.$ By the elementary Morse theory, $M \setminus B^M_{8R_0}(p)$ is diffeomorphic to $\left(M \cap \de B^M_{8R_0}(p) \right)\times [0,\infty)\,.$ Since $\iota$ is proper, $M \cap \de B^M_{8R_0}(p)$ is the union of finite $(n-1)$-dimensional closed manifolds. So $M$ is of finite ends. Fix an end $E$ of $M\,,$ and by the curvature estimate in Lemma \ref{curvature.estimate2},  $r_i^{-1}\iota(E)$ converges locally smoothly  in $\RR^{n+m} \setminus \{0\}$ to an $n$-plane passing $0\in \RR^{n+m}$ as $ r_i\uparrow \infty\,.$ Since $n\geq 3$ and $\SS^{n-1}$ is simply connected,  the multiplicity of $n$-plane is 1. So in the sense of varifolds, $|r_i^{-1}\iota(E)| \weakto |\psi(\RR^n\times \{0^m\} )|$ in $\RR^{n+m} \setminus B_1(0)$ where   $\psi \in SO(n+m). $

$(2)\Rightarrow(3):$ By using a result of Allard and Almgren \cite{allard.radial} and Simon \cite[p273, Theorem 6.6]{leon.isolated},
 outside a compact set, each end  $E$ up to a rotation can be
described as the graph of a vector-valued function over the exterior of a bounded region in $\RR^n\times \{0^m\} \subset \RR^{n+m}$  as \eqref{func.minimal.graph} and
$F$ is defined on $\SS^{n-1} \times [t_1,+\infty) $ satisfying  $\|F\|_{2,\beta}<\infty$ for some $\beta>0$. The result of Allard and
Almgren can be applied since all Jacobi functions of the totally geodesic submanifold $\SS^{n-1} \subset \SS^{n+m-1}$ are Killing vector fields of $\SS^{n-1}$ (see Theorem 5.1.1. in \cite{simons}). Decreasing slightly $\beta$
if necessary, we can assume that $-3\beta \neq \lambda_{k,\pm}$ and apply
Proposition \ref{prop.improv}. So $F = X+R\,,$ where $X$ is  in the kernel of $L$ with
decay between $-\beta$ and $-3\beta$ and $\|R\|_{3\beta}<\infty\,$. If there are no
elements in the kernel of $L$ with decay between $-\beta$ and $-3\beta\,$, we get
$\|F\|_{3\beta}<\infty\,$; in that case we have improved the decay of $F\,$.
So we can iterate this argument until we get the first non-vanishing element in
the kernel. The first decay of elements in the kernel is given by 
$\lambda_{0,+}=-1\,$. Let $\Phi_{k} :\SS^{n-1} \mapsto \RR^m$ denote the  eigenfunction of $\la_{\SS^{n-1}}$ with eigenvalue $ -k(k+n-2)\,,$ and $\Phi_k = (\Phi_k^1,\dots,\Phi_k^m)\,.$ Then
$\Phi_{0}$  is the constant vector-valued function and $F$ can
be written
$$
F(x,t)=e^{-t}b +
R(x,t)\,, \text{ for some } b \in \RR^m\,,
$$ 
with $\|R\|_{1+\eps}<\infty$ for some $\eps>0\,$.
The first term can be interpreted as a translation.  So the translated submanifold
$\iota(E)-b$ can be expressed as the graph of a vector-valued function $G$ over the exterior of a bounded region in
$\RR^{n} \times \{0^m\}$ with the  estimate $\|G\|_{1+\eps}<\infty\,.$

Then, we study the asymptotic behavior of $\iota(E)-b\,.$
By Proposition \ref{prop.improv},
we get the first non-vanishing element in
the kernel. Then the first decay of elements in the kernel is given by
$\lambda_{0,-}=-n+1\,.$ Furthermore, $\lambda_{1,-} = -n,\lambda_{2,-} = -(n+1)$ and $\Phi_{1}^\ell(x) = \sum_{j=1}^nc_{j}^\ell x^j,1\leq \ell\leq m\,.$ Since $n\geq3\,,$  $3(-n+1) <-(n+1)\,.$ By Proposition \ref{prop.improv},
$$G(x,t) = e^{-(n-1)t}a + e^{-nt}\Phi_1(x) + O(e^{-(n+1)t})\, , \text{ for some } a \in \RR^m\,.$$

We have parametrization $$e^tF(x,t) = b+e^{-(n-2)t}a + e^{-(n-1)t}\Phi_1(x) + O(e^{-nt})\,.$$
If $x^1,\cdots ,x^n$
 are coordinates in $\RR^n\times \{0^m\}\,,$ then the graph function $Y=(Y^1,\cdots,Y^m)$ has the
 following asymptotic behavior for $\abs{x}$ large: $$Y^\ell(x)= b^\ell + a^\ell\abs{x}^{2-n} + \sum_{j=1}^nc_{j}^\ell x^j\abs{x}^{-n} +O(\abs{x}^{-n})\, , 1\leq \ell \leq m\,.$$

$(3)\Rightarrow(1):$ If the immersed submanifold $M$ is regular at infinity, then $M$ is of finite embedded ends. So we only need to show that the total curvature of each end is finite. However, this is obvious from the asymptotic behavior.
\end{proof}



 It is well known that if the total curvature is sufficiently small then $M$ must be a plane. See \cite{Moore,Ni} for related results. We include a proof here for the convenience of readers.
\begin{coro} \label{Bernstein1}
     For fixed $\ n,m \in \ZZ^+,n\geq 3,m\geq 1\,,$ if $\iota : M^n \to \RR^{n+m} $  is a complete connected minimal immersion with $\int_{M} \abs{A}^n d\mu_{M} < 2K_0\,,$ then $M $ is a flat $n$-plane in $\RR^{n+m}\,.$
 \end{coro}
\begin{proof}
By Theorem \ref{eqival1}, $M$ is regular at infinity  and $\iota $ is a proper immersion.
So for any fixed $p \in M\,,$ there exists $R_j \uparrow \infty$ such that $\de B_{R_j}(p)$ intersects $M$ transversely. Then we have $\abs{A_M}(p) = 0$ by  curvature estimates in Lemma \ref{curvature.estimate1}, which implies $M$ is a flat $n$-plane in $\RR^{n+m}$.
\end{proof}

Since $\int_M \abs{A}^n d\mu_{M}< \infty\,,$ we can associate a Radon measure $\nu$ on $\RR^{n+m}$ by letting $$\nu(U) = \int_{\iota^{-1}(U)\cap M} \abs{A}^n d\mu_{M}$$ for any open set $U \subset \RR^{n+m}\,.$ As Lemma 2.2 of \cite{Chodosh17}, we show that a sequence of complete immersed minimal submanifolds with uniformly bounded total curvature and volume ratio will have curvature estimates away from at most finitely many points. 

\begin{lemm}\label{lemm:curv.est}
For fixed $I \in \ZZ^+,0 <r_0<R_0< \infty\,,$ suppose that $M_j( \iota_j :M_{j} \to  \RR^{n+m})$ is a sequence of $n$-dimensional complete properly immersed minimal submanifolds with nonempty boundary and $\iota_j(M_j) \subset B_{R_0}(0)$ such that $ \int_{M_j} \abs{A_{M_j}}^n d\mu_{M_j} < IK_0 , \limsup \limits_{j\to \infty}\nu_j(B_{R_0}(0)\setminus B_{\fh{r_0}{2}}(0))< K_0$ and $\vol(B_R(q)) \leq \Lambda R^n$ for any $B_R(q) \subset B_{R_0}(0)$. Then, after passing to a subsequence, we have:
\begin{enumerate}[itemsep=5pt, topsep=5pt]
    \item There exist $C>0$ and a sequence of smooth blow-up sets $\cB_{j} \subset M_{j}$  so that 
\begin{equation} \label{eq:seq.cur.esti}
    |A_{M_{j}}|(x)d(\iota_j(x),\iota_j(\cB_{j} \cup \de M_j) ) \leq C\, , \ |\cB_{j}|< I\, , \ \iota_j(\cB_j) \subset B_{\fh{3}{4}r_0}(0)\, ,
\end{equation}
for all $x \in M_{j}\,.$ 
\item  $\iota_j(\cB_j)$ converges to $ \widetilde  \cB_ \infty \subset \RR^{n+m}$ in the Hausdorff distance sense and  the Radon measure $\nu_j$ converges to $ \nu_\infty$ in the Radon measure sense with $\nu_\infty(p_\infty)\geq 2K_0$ for any $p_\infty \in \widetilde \cB_\infty\,.$
\end{enumerate}
\end{lemm}
\begin{proof}
Firstly, we assume the conclusion in item(1) holds for the fixed $I\,,$ and prove the conclusion in item(2) holds for the same fixed $I\, $.
  Since $\iota_j(\cB_j) \subset B_{\fh{3}{4}r_0}(0) \,,$ after passing to a subsequence, we can assume  $\iota_j(\cB_j)$ converges to $\widetilde \cB_\infty \subset \RR^{n+m}$ in the Hausdorff distance sense, i.e., $d_{\cH}(\iota_j(\cB_j),\widetilde \cB_\infty) \to 0\, $. Since $\nu_j(\RR^{n+m})=\nu_j(B_{R_0}(0)) = \int_{M_j} \abs{A_{M_j}}^n d\mu_{M_j}< IK_0\,,$ after passing to a subsequence, we can assume $\nu_j \to \nu_\infty$ in the Radon measure sense.

  If there exists $p_\infty \in \widetilde \cB_\infty $ with $\nu_\infty(p_\infty)< 2K_0\,,$ then there exists some $\tau_0>0$ small such that $\nu_j(B_{\tau_0}(p_\infty)) <2K_0$ and $\iota_j(p_j) \subset  B_{\fh{\tau_0}{4}}(p_\infty)$ with some $p_j\in \cB_j$ for all $j$ large enough. We denote $\kappa_j : = |A_{M_j}|(p_j)$ and fix $ R_1 > C_1$ where $C_1$ is the constant in Lemma \ref{curvature.estimate1}. By the definition of smooth blow-up sets, after passing to a subsequence, the boundary of $M_j$ in  $B_{\fh{ R_1}{\kappa_j}}(\iota_j(p_j))$ is empty for all $j\,$. Hence $B_{\fh{ R_1}{\kappa_j}}(\iota_j(p_j)) \subset B_{\fh{\tau_0}{2}}(p_\infty)$ and $\nu_j(B_{\fh{ R_1}{\kappa_j}}(\iota_j(p_j)))< 2K_0$ for all $j$ large enough. By Lemma \ref{curvature.estimate1}, $R_1 =\abs{A_{M_j}}(p_j) \fh{ R_1}{\kappa_j} \leq C_1$ for all $j$ large enough, this is a contradiction. So $\nu_\infty(p_\infty)\geq 2K_0$ for any $p_\infty \in \widetilde \cB_\infty\,.$
  
    We will prove \eqref{eq:seq.cur.esti} by induction on $I$. When $I = 1$, the lemma follows the definition of $ K_0$ and Lemma \ref{curvature.estimate1}. Then we assume \eqref{eq:seq.cur.esti} holds for $I-1(I>1)\,.$

   After passing to a subsequence, we may assume that 
$$
\alpha_{j} : = \sup_{x\in M_{j}} |A_{M_{j}}|(x)d(x,\iota_j(\partial M_{j})) \to \infty \, .
$$
If we cannot find such a subsequence, it is easy to see that curvature estimates hold with $\cB_{j} = \emptyset\, $. 

Then  the  standard point picking argument as in Lemma \ref{curvature.estimate1} by passing to a subsequence  allows us to find $\widetilde p_{j}\in M_{j}$ so that $\lambda_{j}:=|A_{M_{j}}|(\widetilde p_{j})\to\infty$ and the rescaled submanifold
$$
\hat M_{j} : = \lambda_{j}(\iota_j(M_{j})- \iota_j(\widetilde p_{j}))
$$
 converges locally smoothly in $\RR^{n+m}$ to a  complete, non-flat, immersed minimal submanifold  $\widehat M_{\infty}$ without boundary.   Moreover $\widehat M_{\infty}$  is of finite total curvature and for all $x \in \widehat M_{\infty}\,,$ 
$$
|A_{\widehat M_{\infty}}|(x) \leq |A_{\widehat M_{\infty}}|(0) = 1\, , \vol(\widehat M_{\infty}\cap B_R(0) ) \leq \Lambda R^n, \text{ for any } R>0\,.$$

Because $\widehat M_{\infty}$ is non-flat,  by Corollary \ref{Bernstein1}, $$\int_{\widehat M_{\infty}}\abs{A_{\widehat M_{\infty}}}^n d\mu_{\widehat M_{\infty}} \geq  2K_0\,.$$ Then there is some radius $ \widehat R > 0$ such that $$\int_{\widehat M_{\infty}\cap B_{\widehat R}(0)} \abs{A_{\widehat M_{\infty}}}^n d\mu_{\widehat M_{\infty}} > \fh{3}{2}K_0\,.$$ By Theorem \ref{eqival1}, $\widehat M_{\infty}$  is regular at infinity. By taking $\widehat R$ larger if necessary, we can assume $\widehat M_{\infty}$ intersects $\partial B_{\widehat{R}}(0)$ transversely and
\begin{equation}\label{eq:infty.cur.est}
|A_{\widehat M_{\infty}}|(x) \leq \fh{1}{4}
\end{equation}
for any $x \in \widehat M_{\infty}\setminus B_{\widehat R}(0)\, $. Then $\int_{M_j \cap B_{\widehat R/\lambda_{j}}(\iota_j(\widetilde p_{j}))} \abs{A_{M_j}}^n > K_0$ for all $j$ large enough while $\limsup \limits_{j\to \infty}\nu_j(B_{R_0}(0)\setminus B_{\fh{r_0}{2}}(0))< K_0\,.$ So after passing to a subsequence, we can assume $\iota_j(\widetilde p_{j}) \in B_{\fh{3}{4}r_0}(0)\,.$ 

We define $\widetilde{M}_{j} : = M_{j} \setminus \iota_j^{-1}\left(B_{\widehat R/\lambda_{j}}(\iota_j(\widetilde p_{j}))\right)\,$. For all $j$ large, due to the choice of $\widetilde p_{j}$ and  $\alpha_{j}\to\infty\,,$ $B_{\widehat R/\lambda_{j}}(\iota_j(\widetilde p_{j})) \cap \de M_j = \emptyset\,.$  Since $\widehat M_{\infty}$ intersects $\partial B_{\widehat{R}}(0)$ transversely,   $M_{j}$ intersects $\partial B_{\widehat R/\lambda_{j}}(\iota_j(\tilde p_{j}))$ transversely.  Thus, $\widetilde{M}_{j}$ is a smooth compact minimal submanifold with smooth, compact boundary \[\partial\widetilde{M}_{j} = \partial M_{j} \cup (\partial B_{\widehat R/\lambda_{j}}(\widetilde \iota_j(p_{j}))\cap M_{j})\, .\]
For all $j$ large, $$ \int_{\widetilde{M}_{j}}\abs{A_{\widetilde{M}_{j}}}^n d\mu_{\widetilde{M}_{j}}< (I-1)K_0\,.$$ By the inductive hypothesis, after passing to a subsequence, there is a sequence of smooth blow-up sets $\widetilde\cB_{j} \subset \widetilde{M}_{j}$ with $|\widetilde\cB_{j}|< I-1\,, \iota_j(\widetilde\cB_{j} ) \subset B_{\fh{3r_0}{4}}(0)$ and a constant $\widetilde C$ (independent of $j$) so that
\begin{equation}\label{eq:curv-est-induct-hyp}
|A_{\widetilde{M}_{j}}|(x)d(\iota_j(x),\iota_j(\widetilde\cB_{j} \cup \partial\widetilde{M}_{j})) \leq \widetilde C
\end{equation}
for all $x\in \widetilde M_j\,.$

We claim that $\cB_{j} : = \widetilde\cB_{j}\cup \{\widetilde p_{j}\}$ is a sequence of smooth blow-up sets for $M_j$. We only need to  check  that none of the points in $\widetilde\cB_{j}$ can appear in the blow-up at $\widetilde p_{j}$ and $\widetilde p_{j}$ cannot appear in the blow-up at any point in $\widetilde\cB_{j}$ (which implies that rescaling  $M_{j}$ around points in $\widetilde\cB_{j}$ still yields a smooth limit). 

We will prove the above claim by contradiction. If 
\begin{gather} \label{lemm.fin.contra}
\liminf_{j\to\infty}\min_{\widetilde q \in \widetilde \cB_{j}} \lambda_{j} |\iota_j(\widetilde q)-\iota_j(\widetilde p_{j})| < \infty\, ,
\end{gather}
then we can assume that the minimum is attained at $\widetilde q_{j}\in\widetilde\cB_{j}$. Since $\widetilde q_{j} \in \widetilde{M}_{j}= M_{j} \setminus \iota_j^{-1}\left( B_{\widehat R/\lambda_{j}}(\iota_j(\widetilde p_{j}))\right)$ and  \eqref{eq:infty.cur.est}, \eqref{lemm.fin.contra} hold,  after passing to a subsequence,
$$
\beta_{j} : = |A_{M_{j}}|(\widetilde q_{j}) \leq \fh{1}{2} |A_{M_{j}}|(\widetilde p_{j}) = \fh{1}{2} \lambda_{j}\, .
$$
Hence, we have
\begin{equation}\label{eq.contra1}
  \liminf_{j\to\infty} \beta_{j}|\iota_j(\widetilde q_{j})-\iota_j(\widetilde p_{j})| < \infty\, . 
\end{equation}

So if the claim is not hold,  \eqref{eq.contra1} must hold.
However,  the blow-up of $\widetilde M_{j}$ around $\widetilde q_{j}$ has no boundary due to the definition of smooth blow-up set, which will contradict with \eqref{eq.contra1}. So $\cB_j$ is a sequence of smooth blow-up sets.

Now, we prove that curvature estimates in \eqref{eq:seq.cur.esti} hold. We argue by contradiction. Suppose that there is $y_{j} \in M_{j}$ such that
\begin{equation} \label{eq:cur.contra}
    \limsup_{j\to\infty}|A_{M_{j}}|(y_{j})d(\iota_j(y_j),\iota_j(\cB_{j}\cup\partial M_{j}))  = \infty\, .
\end{equation}
Combined  \eqref{eq:curv-est-induct-hyp} and the choice of $\widetilde p_{j}$, after passing to a subsequence, we can assume that  $y_{j} \in \widetilde M_{j}$ and
\begin{equation}\label{eq.lem.finite.dist}
\Xi_j : =  d(\iota_j(y_{j}) \,,\iota_j(\cB_{j}\cup\partial M_{j}))  = |\iota_j(y_{j})-\iota_j(\widetilde p_{j})|\, .
\end{equation}
Then, we have
\begin{equation} \label{eq.omega}
\Omega_j := d(\iota_j(y_{j})\, , \iota_j(\widetilde\cB_{j}\cup \partial\widetilde M_{j})) = |\iota_j(y_{j})-\iota_j(\widetilde p_{j})| - \frac{\widehat R}{\lambda_{j}}\,.
\end{equation}
Hence, after passing to a subsequence, we can assume $\Xi_j \to 0\,.$ Otherwise $\Omega_j/\Xi_j \to 1\,,$ then \eqref{eq:cur.contra} will contradict with \eqref{eq:curv-est-induct-hyp}.
Because $\widehat M_{\infty}$ has bounded curvature, so $y_{j}$ cannot appear in the blow-up at $\widetilde p_{j}\,$, i.e.,
\begin{equation} \label{eq.q.blow}
 \liminf_{j\to\infty} \lambda_{j}\Xi_j = \infty\, .   
\end{equation}
Hence, combined  \eqref{eq:curv-est-induct-hyp}, \eqref{eq.omega} and \eqref{eq.q.blow}, we have
\begin{equation} \label{eq.lem.cur.small}
\limsup_{j\to\infty} |A_{M_{j}}|(y_{j}) \frac{\widehat R}{\lambda_{j}} \leq \limsup_{j\to\infty} \frac{\widetilde C \widehat R}{\lambda_{j}d(\iota_j(y_{j}),\iota_j(\widetilde\cB_{j}\cup\partial\widetilde M_{j}))} = 0\, .
\end{equation}
Combined  \eqref{eq:curv-est-induct-hyp}, \eqref{eq.omega}, \eqref{eq.q.blow} and \eqref{eq.lem.cur.small},  we have 
\begin{align*}
\widetilde C &\geq \limsup_{j\to\infty} |A_{M_{j}}|(y_{j})d(\iota_j(y_{j}) ,\iota_j(\widetilde\cB_{j}\cup\partial\widetilde M_{j}))  \\
&= \limsup_{j\to\infty} |A_{M_{j}}|(y_{j}) \left(|\iota_j(y_{j})-\iota_j(\widetilde p_{j})| - \frac{\widehat R}{\lambda_{j}}\right)= \infty\, ,  
\end{align*}
which is a contradiction. So we complete the proof.
\end{proof}


\section{Finiteness of Topology} \label{sec.top}
The following lemma is similar to  Lemma 3.1 in \cite{Chodosh17}, which is crucial for our later arguments. 

\begin{lemm}[Annular decomposition] \label{lemm.Annular.composition}
For fixed $n\,,m\in \ZZ^+\,,n\geq 2\,,$ there is a $0 < \sigma_0 <\fh{1}{2} $ only depending on $n,m$ with the following property. Suppose that $M^n(\iota : M^n \to \bar{B_2(0)} \subset\RR^{n+m})$ is a complete properly immersed submanifold  with $\iota(\de M)  \subset \de B_2(0)\, . $ Assume that for some $\sigma \leq \sigma_0$ and $p \in B_{\sigma_0}(0)\,,$ we have:
\begin{enumerate}[itemsep=5pt, topsep=5pt]
\item For each component $M'$ of $M\,,$ $M' \cap B_{\sigma}(p) \neq \emptyset \,$.
\item The immersed submanifold $M$ intersects $\de B_{\sigma}(p)$ transversely, and $M \cap \de B_{\sigma}(p)$ has $k$ components. Moreover, each component of   $M \cap \de B_{\sigma}(p)$ is diffeomorphic to $\SS^{n-1}$ with the standard smooth structure.
\item The second fundamental form of $M$ satisfies $|A|(x)|\iota(x)-p| \leq \fh{1}{4} $ for all $x \in M \cap \left( \overline{B_{1}(0)} \setminus B_{\sigma}(p) \right)\,.$
\end{enumerate}
Then, $M$ intersects $\de B_{1}(0)$ transversely. Both  $M \cap \left(\overline{B_{1}(0)}\setminus {B_{\sigma}(p)}\right)$ and  $M \cap \de B_{1}(0)$  have $k$ components. Moreover, each component of   $M \cap \de B_{1}(0)$ is diffeomorphic to $\SS^{n-1}$ with the standard smooth structure and
 each component of $M \cap \left(\overline{B_{1}(0)}\setminus {B_{\sigma}(p)}\right)$ is diffeomorphic to $\SS^{n-1}\times [0,1]$ with the standard smooth structure.
\end{lemm}

\begin{proof}
Choose a smooth cutoff function  $\eta \in C^{\infty}_{c}([0,1))$ so that $\eta(r) \in [0,1]\, $, $\eta(r) = 1$ for $r \leq \frac {1}{4}$ and $\eta(r) = 0$ for $r\geq \fh{81}{100}\,.$  We will take $\sigma_0>0$ sufficiently small based on this fixed cutoff function. Let  $\phi(x):=\eta(|\iota(x)|^{2}).$ Consider the function
$$
f(x) = |\iota(x)-p|^{2}\phi(x) + |\iota(x)|^{2}(1-\phi(x))\, .
$$
We see that $f(x) = |\iota(x)-p|^{2}$ if $|\iota(x)|\leq \fh{1}{2}$  and $f(x) = |\iota(x)|^{2}$ if $|\iota(x)| \geq \fh{9}{10}\,.$ For any point $q \in \RR^{n+m}$ and $\xi \in \fX(M)\,,$ we have
\begin{align} \label{lem.top.1}
\nabla(|\iota(x)-q|^{2}) &  = 2(\iota(x)-q)^T,
\end{align}
\begin{gather} \label{lem.top.2}
\begin{aligned}
\nabla^2 (|\iota(x)-q|^{2}) (\xi,\xi) & =  \xi(\xi(|\iota(x)-q|^{2})) - \d_\xi \xi    (|\iota(x)-q|^{2}) \\ &= \db^2(|\iota(x)-q|^{2})(\xi,\xi) + A(\xi,\xi)|\iota(x)-q|^{2} \\ & = 2\left(\abs{\xi}^2 + \ang{A(\xi,\xi),\iota(x)-q}\right)\, ,
\end{aligned}
\end{gather}
\begin{align} \label{lem.top.3}
\d \phi & = \eta'\d |\iota(x)|^2 = 2 \eta' (\iota(x))^T\, ,
\end{align}
\begin{gather} \label{lem.top.4}
\begin{aligned}
\d^2 \phi(\xi,\xi)  &= \eta' \d^2|\iota(x)|^2(\xi,\xi) +  4\eta''\ang{\iota(x),\xi}^2\\
& = 2\eta'\left(\abs{\xi}^2 + \ang{A(\xi,\xi),\iota(x)}\right)+ 4\eta''\ang{\iota(x),\xi}^2\, .
\end{aligned}
\end{gather}
 Combined  \eqref{lem.top.1}, \eqref{lem.top.2}, \eqref{lem.top.3} and \eqref{lem.top.4}, we compute
\begin{align*}
\d^{2} f (\xi,\xi)  = & \phi\d^2 |\iota(x)-p|^{2}(\xi,\xi) +(1-\phi) \d^2|\iota(x)|^{2}(\xi,\xi)\\
&+ 2\ang{\d \phi ,\xi}\ang{\d(|\iota(x)-p|^2),\xi} -2 \ang{\d \phi ,\xi}\ang{\d(|\iota(x)|^2),\xi}\\
&+ |\iota(x)-p|^{2} \d^2\phi(\xi,\xi)
 - |\iota(x)|^{2}\d^2\phi(\xi,\xi)\\
=&2\phi\left(|\xi|^2+ \ang{A(\xi,\xi),\iota(x)-p}\right)
+ 2(1-\phi)\left(|\xi|^2+ \ang{A(\xi,\xi),\iota(x)}\right)\\
& + 8 \eta'  \ang{\iota(x)-p,\xi}  \ang{\iota(x),\xi}  - 8 \eta'  \ang{\iota(x),\xi}^{2}\\
& + 2\eta'(|\iota(x)-p|^{2}-|\iota(x)|^{2})\left(|\xi|^2+ \ang{A(\xi,\xi),\iota(x)}\right)\\ &
+ 4\eta''\left(|\iota(x)-p|^{2}-|\iota(x)|^{2}\right) \ang{\iota(x),\xi}^{2}\\
 =& 2\left(|\xi|^2 + \ang{A(\xi,\xi),\iota(x)-p}\right)
+2(1-\phi)\ang{A(\xi,\xi),p}\\
& - 8 \eta'  \ang{p,\xi} \bangle{\iota(x),\xi} \\
& + 2\eta'\left(|\iota(x)-p|^{2}-|\iota(x)|^{2}\right)\left(|\xi|^2+\ang{A(\xi,\xi),\iota(x)}\right)\\& + 4\eta''\left(|\iota(x)-p|^{2}-|\iota(x)|^{2}\right)\bangle{\iota(x),\xi}^{2}\, .
\end{align*}
If $|\iota(x)| < \fh{1}{2}\,,$  then $x$ is not in the  supports of $1-\phi\, $, $\eta'(\abs{\iota(x)}^2)$ and $\eta''(\abs{\iota(x)}^2)\, $. Hence, by taking $\sigma_0$ sufficiently small, $|A|(x)\leq \fh{5}{8}$ on the supports of $1-\phi\, $, $\eta'(\abs{\iota(x)}^2)$ and $\eta''(\abs{\iota(x)}^2)\,.$ It is easy to see that on $M \cap \left(\overline{B_{1}(0)}\setminus {B_{\sigma}(p)}\right)\,$,
$$
\d^{2}f(\xi,\xi) \geq 2\left(|\xi|^2+ \ang{A(\xi,\xi),\iota(x)-p}\right) - C\abs{p} |\xi|^{2}\, ,
$$
for some $C>0$ only depending on $|\eta|\, ,|\eta'|$ and $|\eta''|\,.$  Since $|A|(x)|\iota(x)-p| \leq \fh{1}{4} $ for all $x \in M \cap \left( \overline{B_{1}(0)} \setminus B_{\sigma}(p) \right)\,,$  we have that 
$$
\d^{2}f(\xi,\xi)  \geq 2\left(\fh{3}{4} -  C\abs{p}\right) |\xi|^{2}\, .
$$
Thus, as long as $\sigma_0 > \abs{p}$ is sufficiently small, $\d^{2}f$ is strictly positive. 

Choosing such a $\sigma_0$, then any critical point of $f$ in $M \cap \left(\overline{B_{1}(0)}\setminus {B_{\sigma}(p)}\right)$ must be a strict local minimum. Suppose $f$ has critical points in  $M \cap \left(\overline{B_{1}(0)}\setminus {B_{\sigma}(p)}\right)\, $.   Since $M$ is properly immersed,  $f$ has finite  critical points on $M \cap \left(\overline{B_{1}(0)}\setminus {B_{\sigma}(p)}\right)\,.$  
We fix a component $M'$ of $M\,.$ By Morse theory\cite[P156]{Hirsch1},  if $f$ has critical points on $M'$, then  $M'$ must have only one critical point $x'$ and 
 $M'\cap f^{-1}(\{x\in M'|f(x)<f(x')\})  = \emptyset\,.$ Then $M' \cap B_{\sigma}(p) = \emptyset\,,$  which is a contradiction. Hence $f$ cannot have any critical points in  $M \cap \left(\overline{B_{1}(0)}\setminus {B_{\sigma}(p)}\right)\, $. The lemma follows from the standard Morse theory.
\end{proof}

 Chodosh proved removal singularity theorem for embedded minimal hypersurfaces with finite total curvature in Euclidean space in his note \cite{Chodosh1}. We adapt his method to study the immersed submanifolds of arbitrary codimension.  
\begin{theo}[Removal singularity]
\label{theo.removal.singularity}
         Suppose that  $\iota: M^{n} \to {B_2(0)}\setminus\{0\} $ is a smooth minimal immersion, i.e.\,,$$\int_M \div_MY d\mu_{M} = 0\, , \text{ for any } Y\in C^\infty_c(B_2(0)\setminus\{0\},\RR^{n+m})\,.$$ The minimal submanifold $M$ satisfies that  $0 \in \bar {\iota(M)}\,,$ $\vol(\iota^{-1}(B_r(0)\setminus\{0\})) \leq \Lambda r^n$ for any $0<r<2$ and $\int_M\abs{A_M}^n d\mu_{M} <\infty\,.$  
         Then there exists a smooth minimal immersion $\hat \iota: \hat M \to {B_2(0)}\, $, i.e., $$\int_{\hat M} \div_MY d\mu_{M} = 0\, , \text{ for any }Y\in C^\infty_c(B_2(0),\RR^{n+m})\,,$$ satisfying that $\hat M \setminus \hat \iota^{-1}(0) = M$ and $\hat \iota_{|M}  =\iota\,.$
\end{theo}
\begin{proof}
    Let $f(x):=|\iota(x)|\,.$  We claim that after rescaling  $M\,,$ we can assume that $|\hat \iota(x)|$ has no critical points in $B_2(0)$ and $M$ intersects $\de B_1(0)$ transversely.
    \begin{lemm} \label{lemm.no.critital.point}
        There exists $\delta_1$ small depending on the immersion $\iota$ such that any critical point of $f$ on $M\cap \iota^{-1}(B_{\delta_1}(0)\setminus \{0\})$ is  a local minimum point.
  
    \end{lemm}
\begin{proof}
   If the lemma does not hold.  There exists a sequence of $  p_j \in M$ satisfying $f(p_j)\to 0$  and $p_j$ is a critical point of $f$ but $p_j$ is not a local minimum point of $f\,.$ Let $\lambda_j = \abs{p_j}^{-1}$ and take $M_j := \lambda_j\iota(M)\,$. By the curvature estimates in Lemma \ref{curvature.estimate2}, after passing to a subsequence,  $M_j$  converges locally smoothly in $\RR^{n+m} \setminus\{0\}$  to $M_\infty$  where $M_\infty$ is an union of $n$-planes. Let $\hat p_j \in M_j $ denote 
   the point corresponding to $p_j$ after rescaling. Let $f_\infty(x) := |\iota_\infty(x)|$ defined on $M_\infty$ and let $f_j(x):=|\iota_j(x)|$ defined on $M_j\,.$ Then $f_j(\hat p_j) = 1$ and $\hat p_j$ is a critical point of $f_j$ but it is not a local minimum point. Since the immersed submanifold $M_j$  converges locally smoothly in $\RR^{n+m} \setminus\{0\}$ to $M_\infty$ and $f_j(\hat p_j) = 1\, $,   after passing to a subsequence, we can assume $\iota_j(\hat p_j)\to \iota_\infty(p_\infty)$ with $ f_\infty(p_\infty)  =1\,.$ 
 Since  $\hat p_j$ is a critical point of $f_j\,,$  $p_\infty$ must be a critical point of $f_\infty$. Thus the immersed submanifold $M_\infty$ must include an  $n$-plane passing the point $\iota_\infty(p_\infty)$ while this plane is normal to the vector $\iota_\infty(p_\infty)\,.$ Hence $p_\infty$ is a 
 local minimum point of $f_\infty\,.$ So for all $j$ large, $\hat p_j$ is a local minimum point of $f_j$ by the locally smooth convergence. This is a contradiction.
\end{proof}
We may rescale $M$ and still denote the immersed submanifold $M$ after rescaling, so that any critical point of $f$ on $M\cap B_{2}(0)$ is a non-degenerate local minimum and $M$ satisfies the assumption in Theorem \ref{theo.removal.singularity}. Let $M'$ be a component of $M$ and $f$ has at least  a critical point on $M'\,.$ Then by Morse theory(as argued in Lemma \ref{lemm.Annular.composition})
 $M'$ is a smooth immersed minimal submanifold in $B_2(0)$ while $0 \not \in \iota(M')\,.$ Hence, by the volume bound and monotonicity formula, there exists $\delta_2$ small such that every component $M'$ of $M \cap B_{
 2}(0)$ with $M' \cap B_{\delta_2}(0) \not = \emptyset$ will have no critical points. Hence we can rescale $M$ so that  $f$ has no critical points and $M$ intersects $\de B_1(0)$ transversely.

 The above arguments are  similar to arguments in \cite[P77]{Chodosh1}. At below, we  use several classical results  in Geometric Measure Theory to address the problems we consider.   Due to the volume bound and monotonicity formula, $M$ has finite components in $B_2(0)\,.$ Let $M'$ be a component of $M$ and let $\eta_j \uparrow \infty \,$. Let $M'_j := \eta_j\iota(M')\,.$
By curvature estimates in Lemma \ref{curvature.estimate2}, after passing to a subsequence, $M_j'$  converges locally smoothly in $\RR^{n+m} \setminus\{0\}$  to  an union of $n$-planes $M_\infty'$ . Since $M'_j$ has no critical points, we can argue as Lemma \ref{lemm.no.critital.point} so that all $n$-planes in $M_\infty'$ are passing the point $0 \in \RR^{n+m}\,.$ Since $M'$ is connected, $f^{-1}(t)\cap M'$ is also connected by the standard Morse theory. Then $M'_\infty$ is an $n$-plane passing $0\in \RR^{n+m}$ and  $f^{-1}(1)\cap M'_j$ converges smoothly  to  $M'_\infty\cap \de B_1(0) = \SS^{n-1}\,.$ So for all $j$ large enough, $f^{-1}(1)\cap M'_j$  is a $K_j$-sheeted covering space of $\SS^{n-1}\,.$ It is well known that  $\SS^{n-1}(n\geq 3)$ is simply connected, so  $K_j = 1$ due to the standard covering space theory. Hence $M'_{\infty}$ is of multiplicity one. Then by Theorem 5 and corresponding Corollary of \cite{asymptotics}, the varifold $|M'_\infty|$ is the unique tangent cone of the varifold $|M'|$ at $0 \in \RR^{n+m}$($|M'_\infty| $ and $|M'|$ denote the the corresponding varifolds of $M'_\infty$ and $ M'$). So we get that $M'$ is a smooth minimal graph passing $0 \in \RR^{n+m}$ in a neighborhood of $0 \in \RR^{n+m}$ due to Allard's regularity theorem\cite{Allard1}. So there exists $\hat \iota$  which extends $\iota$ to a smooth  minimal immersion in $B_2(0)\,$. 
\end{proof}

Inspired by the study in \cite{Chodosh17} of embedded minimal hypersurfaces case, we define the hypothesis \hyperlink{defi:beth}{$(\beth)$} as follows.

Fix $n\,,m\,,I \in \ZZ^+\,, n\geq3\,,m\geq1\,.$  Assume that:

 \begin{enumerate}[itemsep=5pt, topsep=5pt]
 \item We have $M_j(j\in \ZZ^+)$  a sequence of $n$-dimensional complete properly immersed minimal submanifolds in $ {B_{2}(0)}$ with $\iota_j:M_{j}\to \bar{B_{2}(0)} \subset \RR^{n+m}$ and $\iota_j(\partial M_{j})= \partial B_{2}(0) \cap \iota_j (M)\,$.
 \item The submanifolds $M_{j}$ are connected.
 \item The submanifolds have total curvature $\int_{M_j} \abs{A_{M_j}}^n d\mu_{M_j}< IK_0\,.$
 \item The submanifolds satisfy $\vol(M_{j}\cap B_r(p)) \leq \Lambda r^{n}$ for any $B_r(p) \subset B_3(0)\,.$
 \item There is a sequence of non-empty smooth blow-up sets $\cB_{j}\subset B_{\fh{\sigma_{0}}{2}}(0)$ (where $\sigma_{0}$ is fixed in Lemma \ref{lemm.Annular.composition}) with $|\cB_{j}|< I$ and $C>0$ so that for all $x \in M_{j}\,,$
 $$
 |A_{M_{j}}|(x) d(\iota_j(x),\iota_j(\cB_{j}\cup \partial M_{j})) \leq C\,, \text{ for all }j\in \ZZ^+\,,
 $$
 and  $\iota_j(\cB_j)$  converges  to the subset $\widetilde \cB_{\infty} \subset \RR^{n+m}$ of finite points in the sense of Hausdorff distance with $\abs{\widetilde \cB_{\infty}}< I\,.$
 
 \item The submanifold $M_{j} \cap B_{\fh{3}{2}}(0)$ converges in the sense of varifolds in $B_{\fh3 2
 }(0)$ to the union of some complete connected immersed minimal submanifolds $M_\infty = \bigcup_{i=1}^{N} M_{\infty,i}$ in $B_{\fh{3}{2}}(0)$ with corresponding multiplicity, i.e., 
\begin{equation*}
    |M_{j}\cap B_{\fh{3}{2}}(0)| \rightharpoonup  \sum_{i=1}^{N}k_i|M_{\infty, i}|\, .
\end{equation*}
We have $\iota_\infty: M_\infty \to \bar{B_{\fh{3}{2}}(0)}$ and $\iota_\infty(\de M_\infty) =\de B_{\fh{3}{2}}(0)\cap \iota_\infty(M_\infty)\,.$ Moreover, $M_j$ converges locally smoothly in $B_{\fh{3}{2}}(0)\setminus \widetilde \cB_\infty$ to $M_\infty\,.$  
 \end{enumerate}

\noindent Then, we say that the sequence $M_{j}$ satisfies \hyperlink{defi:beth}{$(\beth)$}. 

In item(6), we assume that the limit varifold is the the union of some complete connected immersed minimal submanifolds $M_\infty = \bigcup_{i=1}^{N} M_{\infty,i}$ in $B_{\fh{3}{2}}(0)$ with corresponding multiplicity, while in \cite{Chodosh17}, they  assumed that the limit varifold is a plane  with integer multiplicity.
In our paper, the minimal submanifolds are immersed in $\RR^{n+m},$ so we need to consider more cases.

The following theorem is key to the proof of Theorem \ref{theo.finite.top}. Our proof is inspired by Proposition 7.1 of \cite{Chodosh17}, where they have proved the case of embedded minimal hypersurfaces under the condition of finite index. However, in our situation, the one point concentration may not be a plane, so we need to blow up again and use an induction argument on the total curvature of the minimal submanifolds.


\begin{theo} \label{prop.diffeo}
Given a sequence $M_{j}$ satisfying \hyperlink{defi:beth}{$(\beth)$} and each $M_j$  intersects $\partial B_{1}(0)$ transversely. By passing to a subsequence, all of the $M_{j}\cap B_{1}(0)$ are diffeomorphic. 
\end{theo}

\begin{proof}

We prove the theorem by induction on $I$. For $I=1$ the theorem trivially follows from curvature estimates in Lemma \ref{curvature.estimate1}, the volume bound and Lemma \ref{lemm:covers}. 

Then we suppose the theorem \ref{prop.diffeo} holds for $I-1(I>1)\,.$   We choose $r_0$ small enough so that  for any $\widetilde p_\infty \in \tilde \cB_\infty\,,$ $M_\infty$ is sufficiently smoothly close to the union of $n$-planes passing $\widetilde p_\infty$ in $B_{r_0}(\widetilde p_\infty)\,,$ and $\min \limits_{\substack{\widetilde p_\infty,\widetilde q_\infty \in \tilde \cB_\infty \\ \widetilde p_\infty \not=\widetilde q_\infty}}|\widetilde p_\infty- \widetilde q_\infty|>4r_0\, . $ Since $M_j$  converges locally smoothly in $B_{\fh{3}{2}}(0) \setminus \tilde \cB_\infty$ to $M_\infty\,,$  fixing $r_0$ small enough, for all $j$ large enough,  $M_j \cap B_{1}(0) \setminus B_{r_0}(\tilde \cB_\infty)$ is a smooth covering space of $M_\infty \cap B_{1}(0) \setminus B_{r_0}(\tilde \cB_\infty)$ with the numbers of sheets on different components of $M_\infty \cap B_{1}(0) \setminus B_{r_0}(\tilde \cB_\infty)$   uniformly bounded due to the monotonicity formula and the uniform volume bound. So after passing to a subsequence, we can assume that all of $M_j\cap B_{1}(0) \setminus B_{r_0}(\tilde \cB_\infty)$ are diffeomorphic  by  Lemma \ref{lemm:covers}. From Theorem 7.6.2  and Proposition 7.6.4 of \cite{Amiya}, the smooth structure is independent of gluing maps. So there are only finitely many ways to connect the regions of $M_j\cap B_{1}(0) \setminus B_{r_0}(\tilde \cB_\infty)$ to the regions of $M_j \cap B_{r_0}(\tilde \cB_\infty)$, and we only need to prove that after passing to a subsequence, $M_j \cap B_{r_0}(\widetilde p_\infty)  $ are diffeomorphic for every $\widetilde p_\infty \in \tilde \cB_\infty\,.$ After rescaling, $$\check M_j :=  r_0^{-1}(
\iota_j(M_j) - \widetilde p_\infty)\,.$$  By the volume bound and monotonicity formula, we can choose a component of $\check M_j$ in $B_2(0)\,,$ and we get a sequence of minimal submanifolds in $\RR^{n+m}$ satisfying \hyperlink{defi:beth}{$(\beth)$} with limit submanifold in item (6) is smoothly  sufficiently close to the union of $n$-planes passing $0 \in \RR^{n+m}\,.$ Abusing notation slightly, we will still denote this sequence of submanifolds $M_j\,.$ Hence we only need to prove the theorem for this sequence of minimal submanifolds.

If $|\tilde \cB_\infty| \geq 2\,,$ we can argue as the proof of Lemma \ref{lemm:curv.est},  and get that the new sequence $M_j$ as defined above has total curvature $$\int_{M_j} \abs{A_{M_j}}^n d \mu_{M_j}< (I-1)K_0\,.$$ So the theorem holds by the induction hypothesis.

If $|\tilde \cB_\infty| = 1 $ and $\liminf \limits_{j\to \infty}|\cB_{j}| = 1\,,$  after passing to a subsequence, we can write $\cB_{j} = \{p_{j}\}\,,$  $\widetilde \cB_{\infty}=\{\widetilde p_{\infty}\}$ and $\lambda_{j} : = |A_{M_{j}}|(p_{j})\,$. By passing to a subsequence, we have that 
$$
\breve M_{j} : = \lambda_{j}(\iota_j(M_{j})-\iota_j(p_{j}))
$$
converges to  a complete, non-flat, properly immersed minimal submanifold $\breve M_{\infty}\subset \RR^{n+m}$ without boundary satisfying that the total curvature of $\breve M_{\infty}$  $\leq$ $IK_0$ and $\vol(\breve{M}_{\infty} \cap B_{r}(0))\leq \Lambda r^{n}$ for any $r>0$ due to the monotonicity formula. So by Theorem \ref{eqival1}, $\breve{M}_{\infty}$ is  regular at infinity. In particular, we can choose $R>0$ so that $\breve{M}_{\infty}$ intersects $\partial B_{R}(0)$ transversely and
$$
|A_{\breve{M}_{\infty}}|(x) |\iota_\infty(x)| < \frac 1 4
$$
for all $x \in\breve{M}_{\infty}\setminus B_{R}(0)$, where $|\iota_\infty(x)|$ denotes the Euclidean distance between $0 \in \RR^{n+m}$ and the image of $x \in \breve M_\infty$ in $\RR^{n+m}\,.$

Case I: after passing to a subsequence,   if for all $j\,,$ 
\begin{equation}\label{eq:prop.cur1}
|A_{M_{j}}|(x)|\iota_j(x)-\iota_j(p_{j})| < \frac 1 4
\end{equation}
for all $x \in M_{j} \cap \left( B_{2}(0) \setminus B_{R/\lambda_{j}}(\iota_j(p_{j}))\right)\,$. Then Lemma \ref{lemm:covers} and Lemma \ref{lemm.Annular.composition} imply that after passing to a subsequence, all of the submanifolds $M_{j}\cap B_{1}(0)$ are diffeomorphic (here, we have used the fact that the ends of $\breve M_\infty$ are diffeomorphic to $\SS^{n-1}\times (0,1)$ with the standard smooth structure and $\breve M_j\cap B_{R}(0)$ is a smooth covering space of $\breve M_\infty \cap B_{R}(0)\,$).

Case II: on the other hand, if \eqref{eq:prop.cur1} does not hold, we may choose $\delta_{j}$ to be the smallest radius
   greater than $R/\lambda_{j}$ so that
$$
|A_{M_{j}}|(x)|\iota_j(x)- \iota_j(p_{j})| < \fh{1}{4}
$$
for all $x \in M_{j} \cap \left( B_{2}(0) \setminus B_{\delta_{j}}(\iota_j(p_{j}))\right)\,$. For all $j$ sufficiently large, such a $\delta_{j}$ exists with $\delta_{j}\to 0$. This follows from the fact that $M_{j}$ converges smoothly away from $\tilde p_{\infty}$ to a submanifold sufficiently close to the union of $n$-planes. Furthermore, we may assume $\liminf \limits_{j\to \infty
}\lambda_j\delta_j = \infty\,,$ otherwise we can take a subsequence of $M_j$ and take $R$ larger so that $R > \liminf \limits_{j\to \infty
}\lambda_j\delta_j\,,$ then \eqref{eq:prop.cur1} holds for all $x \in M_{j} \cap \left( B_{2}(0) \setminus B_{R/\lambda_{j}}(\iota_j(p_{j}))\right)$ with all  $j$ large enough, and the theorem follows the same arguments in Case I. At below, we define
$$
\hat M_{j} := \delta_{j}^{-1} (\iota_j(M_{j})-\iota_j(p_{j}))\, .
$$
After passing to a subsequence, there is an immersed minimal submanifold $\hat M_{\infty}$ in $ \RR^{n+m} \setminus \{0\}$ so that $\hat M_{j}$ converges locally smoothly in  $\RR^{n+m}\setminus \{0\}$ to $\hat M_{\infty}$  with finite multiplicity(the multiplicity may be different for distinct components of $\hat M_\infty$) by the curvature estimates in item (5). Moreover, after passing to a subsequence, $|\hat M_{j}| \weakto \abs{\hat M_\infty}$ in the sense of varifolds in $B_1(0)\,.$ 

Since $\hat M_{\infty}$ has finite total curvature and satisfies the volume growth condition of Theorem \ref{theo.removal.singularity} by the varifold convergence and monotonicity formula,  the possible singularity at $\{0\}$ is removable. Hence $\hat M_{\infty}$ is an immersed minimal submanifold in $\RR^{n+m}$ with total curvature $\int_{\hat M_{\infty}} \abs{A_{\hat M_{\infty}}}^n d\mu_{\hat M_{\infty}}\leq I K_0$ and $\vol(\hat M_{\infty}\cap B_{r}(0)) \leq \Lambda r^{n}$ for any $r>0\,.$ Moreover, $\hat M_{\infty}$  has bounded number of components due to the volume bound. By Theorem \ref{eqival1}, every component of $\hat M_{\infty}$ is regular at infinity.  Then we can choose $\gamma \geq 1$ large enough so that $\partial B_{\gamma}(0)$ intersects each component of $\hat M_{\infty}$ transversely, and each component of  $\hat M_{\infty} \cap \partial B_{\gamma}(0)$ is diffeomorphic to $\SS^{n-1}$ with the standard smooth structure. Moreover the number of components of $\hat M_{\infty} \cap \partial B_{\gamma}(0)$ is no more than $\Lambda\,.$ By the choice of $\delta_{j}\,$, the curvature estimates \eqref{eq:prop.cur1} hold for all $x \in  M_{j} \cap \left( B_{2}(0) \setminus B_{\gamma \delta_{j}}(\iota_j(p_{j}))\right)\,$. Then by applying Lemma \ref{lemm.Annular.composition}, we see that $M_{j} \cap \left( B_{2}(0) \setminus B_{\gamma \delta_{j}}(\iota_j(p_{j}))\right)$ is diffeomorphic to the union of annular regions. In particular, $M_{j}\cap B_{\gamma\delta_{j}}(\iota_j(p_{j}))$ must be connected (because we have assumed that $M_{j}$ is connected in \hyperlink{defi:beth}{$(\beth)$}). Then we only need to prove $\hat M_j \cap B_\gamma(0)$ are diffeomorphic to each other after passing to a subsequence.

 By the choice of $\delta_j\,,$ there exists at least a non-flat component of $\hat M_{\infty}\,.$ Applying Corollary \ref{Bernstein1}, we can take $\gamma$ sufficiently large and $\kappa$ sufficiently small so that $$\int_{ \hat M_{\infty}\cap B_\gamma(0)\setminus B_{3\kappa}(0)} \abs{A_{\hat M_{\infty}}}^n d\mu_{\hat M_{\infty}}\geq \fh{3}{2}K_0\,.$$
Then after passing to a subsequence, $$\int_{\hat M_j \cap B_{2\kappa}(0) }\abs{A_{\hat M_j}}^n d\mu_{\hat M_j} < (I-1)K_0$$ for all $j\,.$ 
After rescaling $\widetilde M_j := \kappa^{-1} \hat \iota_j(\hat M_j)\,.$ We can choose some  component $\widetilde M_j'$ of $\widetilde M_j\,,$ then we have the fact  that the sequence $\widetilde M_j'$ will satisfy \hyperlink{defi:beth}{$(\beth)$} with total curvature $$\int_{\widetilde M_j'\cap B_2(0)} \abs{A_{\widetilde M_j'}}^n d\mu_{\widetilde M'_j} <(I-1)K_0\,.$$
By the induction hypothesis, $\widetilde M_j'\cap B_1(0)$ are diffeomorphic. Then after passing to a subsequence, $\hat M_j\cap B_\kappa(0)$ are diffeomorphic. By the locally smooth convergence of $\hat M_j\,,$ we have $\hat M_j \cap B_\gamma(0)$ are diffeomorphic to each other after passing to a  subsequence. This completes the proof in the case that $|\tilde \cB_\infty| = 1 $ and $\liminf \limits_{j\to \infty}|\cB_{j}| = 1\,.$

If $|\tilde \cB_\infty| =1\, ,\liminf\limits_{j\to \infty}|\cB_j| \geq 2\, $. Then $\eps_j : = \max \limits_{\substack{p_i,q_j \in \cB_j \\ p_j \not = q_j}}d (\iota_j(p_j),\iota_j(q_j)) \to 0\,.$ By the definition of smooth blow-up sets, $\lim \limits_{j\to \infty} \eps_j |A_{M_j}|(p_j) = \infty$ for all $p_j \in \tilde \cB_j\,.$ Then we fix $p_j\,,q_j \in \cB_j$  satisfying  $\eps_j = |\iota_j(p_j)-\iota_j(q_j)|\,.$  Let  $ M^*_j: = \fh{\sigma_0}{4\eps_j}(\iota_j(M_j)-\iota_j(p_j))\,,$ and we can choose a component $M^{*'}_j$ of $ M^*_j$ in $B_2(0)\,.$  After passing to a subsequence, this sequence will satisfy \hyperlink{defi:beth}{$(\beth)$} with $|\tilde \cB_{\infty}| \geq 2$ or $$\int_{M^{*'}_j\cap B_2(0)} \abs{A_{M^{*'}_j}}^nd\mu_{M^{*'}_j} <(I-1)K_0\,.$$


 Hence after passing to a subsequence, all of the $B_{\fh{4\eps_j}{\sigma_0}}(\iota_j(p_j)) \cap M_j$ are diffeomorphic. As the situation of $|\tilde \cB_\infty| = 1 $ and $\liminf \limits_{j\to \infty}|\cB_{j}| = 1\,,$ then we can argue in two cases and prove the theorem.
 
\end{proof}

 \begin{proof}[Proof of Theorem \ref{theo.finite.top}]
       We will prove Theorem \ref{theo.finite.top} by contradiction. Since the volume bound and monotonicity formula imply that there exist at most finite number of components for any minimal submanifold satisfying the assumption of Theorem \ref{theo.finite.top}, without loss of generality, we can assume the minimal submanifolds satisfying the assumption of Theorem \ref{theo.finite.top} are connected. If $M_j^n$ is a sequence of pairwise non-diffeomorphic  complete connected, immersed minimal submanifold in $\RR^{n+m}$ with $\vol(M_j\cap B_{R}(0)) \leq \Lambda R^{n}$ for any $R>0$ and $$\int_{M_j} \abs{A_{M_j}}^n d \mu_{M_j}\leq \Gamma < IK_0\,\,.$$ By rescaling $M_j,$ we can assume $$\int_{M_j\setminus B_{\fh{1}{j}}(0)}\abs{A_{M_j}}^n d \mu_{M_j} <\fh{1}{j}\,,$$ and $M_j $ intersects $\de B_1(0)$ transversely. By 
     Theorem \ref{eqival1},
      $M_j$ is properly immersed and  regular at infinity. By rescaling $M_j\,,$ we can assume   $M_j \setminus B_{\fh1 2}(0) $ is the union of minimal graph and each minimal graph is defined over the exterior of a bounded region in an $n$-plane passing $0 \in \RR^{n+m}\,.$  So after passing to a subsequence, we can assume the $M_j\cap B_1(0) $ are pairwise non-diffeomorphic. By Lemma \ref{lemm:curv.est}, the sequence $M_j\cap B_2(0)$ satisfies  \hyperlink{defi:beth}{$(\beth)$}. Then by Theorem \ref{prop.diffeo}, after passing to a subsequence, all of the $M_j\cap B_1(0)$ are diffeomorphic. This is a contradiction.
 \end{proof}

\appendix

\section{Calculations of Minimal Surface System}\label{Appendix.A}
 We can choose a local coordinate $(U,y^1,\dots,y^{n-1})$ for a neighborhood of $\SS^{n-1} \subset \RR^{n}$ where $\SS^{n-1}$ is of the standard Riemannian metric as a sphere  with radius 1\,.  Hence, in this coordinate, the Riemannian metric is $g_{\SS^{n-1}} = \sigma_{ij}dy^idy^j\,.$ Let $x:\SS^{n-1}\to \RR^n,$ and $x = (x^1, \cdots,x^n).$ We denote $e_i : = \di x \in \RR^n$ for $1\leq i \leq (n-1),$ where $\de_i x = (\fh{\de x^1}{\de y^i},\cdots, \fh{\de x^n}{\de y^i}).$  Then $ \ang{e_i,x} = 0\, , \ang{e_i,e_j} = \sigma_{ij}\, , $ and $\{e_1,\dots , e_{n-1},x\}$ is a local frame on $\RR^n \,.$ we denote $\sigma^{ij}$  the $ij$th entry in the inverse of  
   $(\sigma_{ij})\, .$ We have $$\di e_j=\di(\dj x) = \fh{\de}{\de y^i}(\fh{\de x^1}{\de y^j},\cdots, \fh{\de x^n}{\de y^j}) = (\fh{\de^2 x^1}{\de y^i\de y^j},\cdots, \fh{\de^2 x^n}{\de y^i\de y^j})\, .$$  Then $\di e_j = \ang{\di e_j,x}x + \ang{\di e_j,e_k} \sigma^{k\ell}e_\ell = -\sigma_{ij}x + \ang{\di e_j,e_k} \sigma^{k\ell}e_\ell\, .$
 Given a function $F$ on  Riemannian   manifold $(\SS^{n-1} \times \RR^+,g_{\SS^{n-1}}+ dt^2)\, ,$ we denote $\d F$ as the gradient of $F\, , $ and $\d^2F$ as the Hessian of $F\, \,.$  

   We compute the induced Riemannian metric from $\RR^{n+m}$ of the graph in \eqref{func.minimal.graph}. In local coordinate $(U \times (a,b),y^1,\dots,y^{n-1},t)\, ,$ let $F_i = (\fh{\de F^1}{\de y^i}, \cdots, \fh{\de F^m}{\de y^i}),$ and $F_t =(\fh{\de F^1}{\de t}, \cdots, \fh{\de F^m}{\de t}).$ 
   \begin{align*}
       g_{ij} &= e^{2t} (\sigma_{ij}+\ang{F_i,F_j})\, ,1\leq i,j \leq(n-1)\, , \\
       g_{nn} &=e^{2t}(1+ \abs{F_t+ F}^2)\, , \\g_{ni} &= e^{2t}\ang{F_t+ F,F_i}\, ,1\leq i \leq (n-1)\, .
   \end{align*} We denote $g = \det(g_{ij})$ and denote $g^{ij}$ 
   the $ij$th entry in the inverse of  
   $(g_{ij})\, \,.$ Let $\cQ(F)$ denote the nonlinear term about $F\, ,\d F\, \,.$ Hence 
   \begin{gather}
   \begin{aligned}
       g^{ij} &=  e^{-2t}(\sigma^{ij} + \cQ(F))\, ,1\leq i\,,j \leq(n-1)\,, \\
       g^{nn} &=e^{-2t}(1 + \cQ(F))\,,\\
       g^{ni} &= e^{-2t}\cQ(F)\, ,1\leq i \leq (n-1)\,.
       \end{aligned} \label{eq.inver.asym}
   \end{gather}
Let $\la$ denote the Laplacian operator on the graph in \eqref{func.minimal.graph} with the induced Riemannian metric from $\RR^{n+m}\,.$ Combined  \eqref{func.minimal.graph}, we have
   \begin{equation*}
       \la \Psi(x,t) =  0 \, .
   \end{equation*}
Hence, \begin{equation*}
    \la (e^tx) = 0\, , \text{ and }\la (e^tF) = 0 \, .
\end{equation*}
 In local coordinate, we have
   \begin{align}
       \dt(\sqrt{g}g^{nn}e^tx) +\dt(\sqrt{g}g^{nj}e^t\dj x) + \di(\sqrt{g}g^{ij}e^t\dj x) +
       \dj(\sqrt{g}g^{jn}e^tx)= 0 \label{eq.minimal.domain} \, ,\\ \dt(\sqrt{g}g^{nn}e^t(F+F_t))+\dt(\sqrt{g}g^{nj}e^tF_j) + \di(\sqrt{g}g^{ij}e^tF_j) +
       \dj(\sqrt{g}g^{jn}e^t(F+F_t))= 0 \, .\label{eq.minimal.image}
   \end{align}

From \eqref{eq.minimal.domain}, we have\begin{gather}
\begin{aligned}
    \dt(\sqrt{g}g^{nn}e^t)x +\dt(\sqrt{g}g^{nj}e^t)e_j + \di(\sqrt{g}g^{ij}e^t)e_j &\\
    +\sqrt{g}g^{ij}e^t\di(e_j)+
       \dj(\sqrt{g}g^{jn}e^t)x + \sqrt{g}g^{jn}e^t e_j & = 0 \, . 
       \end{aligned}        
\end{gather}
 So we have 
 \begin{gather}
     \begin{aligned}
     \dt(\sqrt{g}g^{nn}e^t) +\dj(\sqrt{g}g^{jn}e^t) - \sqrt{g}g^{ij}e^t\sigma_{ij} = 0 \, , \\
     \dt(\sqrt{g}g^{nj}e^t) + \di(\sqrt{g}g^{ij}e^t) +
     \sqrt{g}g^{k\ell}e^t\ang{\de_k e_\ell,e_i} h^{ij}+
     \sqrt{g}g^{jn}e^t = 0 \, . \label{eq.iden.1}
     \end{aligned}
 \end{gather}
 From \eqref{eq.minimal.image},
 we have \begin{gather}
       \begin{aligned}
    &\dt(\sqrt{g}g^{nn}e^t)(F+F_t)+ \sqrt{g}g^{nn}e^t(F_t+F_{tt}) +\dt(\sqrt{g}g^{nj}e^t)F_j + \sqrt{g}g^{nj}e^t\dt F_j+
    \\
    &\di(\sqrt{g}g^{ij}e^t)F_j +
    \sqrt{g}g^{ij}e^t\di F_j+
       \dj(\sqrt{g}g^{jn}e^t)(F+F_t)+
\sqrt{g}g^{jn}e^t\dj(F+F_t) = 0 \, .
       \end{aligned} \label{eq.sim.min.doma}
\end{gather}
Combined \eqref{eq.iden.1} and \eqref{eq.sim.min.doma}, we have
\begin{gather}
    \begin{aligned}
         \sqrt{g}g^{nn}e^t(F_t+F_{tt})  +\sqrt{g}g^{ij}e^t\di F_j+\sqrt{g}g^{ij}e^t\sigma_{ij}(F+F_t)& 
    \\
-\sqrt{g}g^{k\ell}e^t\ang{\de_k e_\ell,e_i} \sigma^{ij}F_j +
\sqrt{g}g^{jn}e^t\dj F_t+ \sqrt{g}g^{nj}e^t\dt F_j &= 0 \, .
    \end{aligned}
\end{gather}
Combined  \eqref{eq.inver.asym}, we have
\begin{gather}
    \begin{aligned}
       \sqrt{g}e^{-t}  \left( F_t+F_{tt}  +\sigma^{ij}\di F_j+
       (n-1)(F+F_t) -\sigma^{k\ell}\ang{\de_k e_\ell,e_i} \sigma^{ij}F_j +
\cQ(F) \right )= 0 \, ,
    \end{aligned}
\end{gather}
where $\cQ(F)$ gathers all the nonlinear terms consisting of $F\,, \d F\,, \d^2F$ at least  cubic.
 Moreover, $$\sigma^{ij}\di F_j-\sigma^{k\ell}\ang{\de_k(e_\ell),e_i} \sigma^{ij}F_j = \sigma^{ij}\left(\di(\dj F) -\ang{\de_i(e_j),e_k} \sigma^{k\ell}F_\ell\right)
  = \la_{\SS^{n-1}}F \, \,.$$ So we have equation \eqref{eq.minimal.graph}
\begin{gather*}
    F_{tt} + nF_t +(n-1)F + \la_{\SS^{n-1}}F +\cQ(F) = 0 \,.
\end{gather*}

\section{Topology Lemma on Covering Space} \label{Appendix.B}
\begin{lemm}\label{lemm:covers}
    Let $M$ be a smooth compact $n$-manifold with boundary$($possibly empty$)$  and $k\in\ZZ^+$. Then there exist at most $N = N(\pi_1(M),k)($possibly disconnected$)$ pairwise non-diffeomorphic smooth $k$-sheeted covering spaces of $M$.
\end{lemm}

\begin{proof}
    Since $M$ is a  smooth compact $n$-manifold, $M$ and some CW-complex  with finite cells are homotopy equivalent(see \cite{Hatcher} for more details about CW-complex, fundamental group and covering space). Hence $\pi_1(M)$ is finitely generated. From \cite[P70]{Hatcher}, we know that $k$-sheeted covering spaces of $M$ are classified by equivalence classes of homomorphisms $\pi_1(M) \to \Sigma_k\,,$ where $\Sigma_k$ is the symmetric group on $k$ symbols and the equivalence relation identifies a homomorphism $\rho$ with each of its conjugates $h^{-1}\rho h $ by elements $h \in \Sigma_k\,$.   Since $\pi_1(M)$ is finitely generated, a homomorphism is determined by the image of the $\ell(\in \NN)$ generators of $\pi_1(M)\, \,.$ Hence there are at most $(k!)^\ell$ such homomorphisms. It implies that there are at most $(k!)^\ell$ equivalence classes of homomorphisms $\pi_1(M) \to \Sigma_k\,,$ and  at most $(k!)^\ell$ pairwise non-diffeomorphic smooth $k$-sheeted covering spaces of $M.$
\end{proof}

\ 

\bigskip

 {\footnotesize

\ 
 
\noindent\textbf{Acknowledgment:} 
The first author is partially supported by NSFC 12371053. The authors wish to express
their sincere gratitude to Otis Chodosh for his interest and valuable comments.
\medskip

\ 

\noindent\textbf{Data availability:} No datasets were generated or analysed during the current study.
 
}

\bigskip

\noindent\textbf{\large Declarations} 

 \
 
 {\footnotesize

\noindent\textbf{Conflict of interest:} The authors declare that they have no conflict of interest. 
\medskip

\medskip

 }

\bibliographystyle{alpha}
\bibliography{ref}

@article{anderson1984compactification,
  TITLE={The compactification of a minimal submanifold in Euclidean space by the Gauss map},
  AUTHOR={Anderson, Michael T},
  YEAR={1984},
  JOURNAL={(Preprint)}
}

@article {Chodosh17,
    AUTHOR = {Chodosh, Otis and Ketover, Daniel and Maximo, Davi},
     TITLE = {Minimal hypersurfaces with bounded index},
   JOURNAL = {Invent. Math.},
  FJOURNAL = {Inventiones Mathematicae},
    VOLUME = {209},
      YEAR = {2017},
    NUMBER = {3},
     PAGES = {617--664},
      ISSN = {0020-9910,1432-1297},
   MRCLASS = {53C42 (53C21)},
  MRNUMBER = {3681392},
MRREVIEWER = {Christine\ Breiner},
       DOI = {10.1007/s00222-017-0717-5},
       URL = {https://doi.org/10.1007/s00222-017-0717-5},
}

@article {simons,
    AUTHOR = {Simons, James},
     TITLE = {Minimal varieties in riemannian manifolds},
   JOURNAL = {Ann. of Math. (2)},
  FJOURNAL = {Annals of Mathematics. Second Series},
    VOLUME = {88},
      YEAR = {1968},
     PAGES = {62--105},
      ISSN = {0003-486X},
   MRCLASS = {53.04 (35.00)},
  MRNUMBER = {233295},
MRREVIEWER = {W.\ F.\ Pohl},
       DOI = {10.2307/1970556},
       URL = {https://doi.org/10.2307/1970556},
}

@article {Tysk,
    AUTHOR = {Tysk, Johan},
     TITLE = {Finiteness of index and total scalar curvature for minimal
              hypersurfaces},
   JOURNAL = {Proc. Amer. Math. Soc.},
  FJOURNAL = {Proceedings of the American Mathematical Society},
    VOLUME = {105},
      YEAR = {1989},
    NUMBER = {2},
     PAGES = {429--435},
      ISSN = {0002-9939,1088-6826},
   MRCLASS = {53C42 (58C40 58E15)},
  MRNUMBER = {946639},
MRREVIEWER = {M.\ Elisa G. G. de Oliveira},
       DOI = {10.2307/2046961},
       URL = {https://doi.org/10.2307/2046961},
}

@article{Moore,
 ISSN = {00222518, 19435258},
 URL = {http://www.jstor.org/stable/24899166},
  author = {Helen Moore},
 journal = {Indiana University Mathematics Journal},
 number = {4},
 pages = {1021--1043},
 publisher = {Indiana University Mathematics Department},
 title = {Minimal Submanifolds with Finite Total Scalar Curvature},
 urldate = {2025-07-19},
 volume = {45},
 year = {1996}
}

@article {Ni,
    AUTHOR = {Ni, Lei},
     TITLE = {Gap theorems for minimal submanifolds in {${\bf R}^{n+1}$}},
   JOURNAL = {Comm. Anal. Geom.},
  FJOURNAL = {Communications in Analysis and Geometry},
    VOLUME = {9},
      YEAR = {2001},
    NUMBER = {3},
     PAGES = {641--656},
      ISSN = {1019-8385,1944-9992},
   MRCLASS = {53C42 (53C40)},
  MRNUMBER = {1895136},
MRREVIEWER = {Yi\ Bing\ Shen},
       DOI = {10.4310/CAG.2001.v9.n3.a2},
       URL = {https://doi.org/10.4310/CAG.2001.v9.n3.a2},
}

@book {xin,
    AUTHOR = {Xin, Yuanlong},
     TITLE = {Minimal submanifolds and related topics},
    SERIES = {Nankai Tracts in Mathematics},
    VOLUME = {16},
   EDITION = {Second},
 PUBLISHER = {World Scientific Publishing Co. Pte. Ltd., Hackensack, NJ},
      YEAR = {2019},
     PAGES = {xvi+380},
      ISBN = {978-981-3236-05-9},
   MRCLASS = {53C42 (53-02 53A10 53C40)},
  MRNUMBER = {3837570},
}

@article {allard.radial,
    AUTHOR = {Allard, William K. and Almgren, Jr., Frederick J.},
     TITLE = {On the radial behavior of minimal surfaces and the uniqueness
              of their tangent cones},
   JOURNAL = {Ann. of Math. (2)},
  FJOURNAL = {Annals of Mathematics. Second Series},
    VOLUME = {113},
      YEAR = {1981},
    NUMBER = {2},
     PAGES = {215--265},
      ISSN = {0003-486X},
   MRCLASS = {49F22 (53A10 58E12)},
  MRNUMBER = {607893},
MRREVIEWER = {E.\ Giusti},
       DOI = {10.2307/2006984},
       URL = {https://doi.org/10.2307/2006984},
}

@article {Allard1,
    AUTHOR = {Allard, William K.},
     TITLE = {On the first variation of a varifold},
   JOURNAL = {Ann. of Math. (2)},
  FJOURNAL = {Annals of Mathematics. Second Series},
    VOLUME = {95},
      YEAR = {1972},
     PAGES = {417--491},
      ISSN = {0003-486X},
   MRCLASS = {49F20},
  MRNUMBER = {307015},
MRREVIEWER = {M.\ Klingmann},
       DOI = {10.2307/1970868},
       URL = {https://doi.org/10.2307/1970868},
}

@book {leon.isolated,
     TITLE = {Harmonic mappings and minimal immersions},
    SERIES = {Lecture Notes in Mathematics},
    VOLUME = {1161},
    EDITOR = {Giusti, E.},
      NOTE = {Lectures given at the first 1984 session of the Centro
              Internationale Matematico Estivo (CIME) held at Montecatini,
              June 24--July 3, 1984},
 PUBLISHER = {Springer-Verlag, Berlin},
      YEAR = {1985},
     PAGES = {viii+285},
      ISBN = {3-540-16040-X},
   MRCLASS = {58E20},
  MRNUMBER = {821967},
       DOI = {10.1007/BFb0075135},
       URL = {https://doi.org/10.1007/BFb0075135},
}

@article {Mazet1,
    AUTHOR = {Mazet, Laurent},
     TITLE = {Minimal hypersurfaces asymptotic to {S}imons cones},
   JOURNAL = {J. Inst. Math. Jussieu},
  FJOURNAL = {Journal of the Institute of Mathematics of Jussieu. JIMJ.
              Journal de l'Institut de Math\'ematiques de Jussieu},
    VOLUME = {16},
      YEAR = {2017},
    NUMBER = {1},
     PAGES = {39--58},
      ISSN = {1474-7480,1475-3030},
   MRCLASS = {53A10},
  MRNUMBER = {3591961},
MRREVIEWER = {Jianquan\ Ge},
       DOI = {10.1017/S1474748015000110},
       URL = {https://doi.org/10.1017/S1474748015000110},
}

@book {GMT,
    AUTHOR = {Simon, Leon},
     TITLE = {Lectures on geometric measure theory},
    SERIES = {Proceedings of the Centre for Mathematical Analysis,
              Australian National University},
    VOLUME = {3},
 PUBLISHER = {Australian National University, Centre for Mathematical
              Analysis, Canberra},
      YEAR = {1983},
     PAGES = {vii+272},
      ISBN = {0-86784-429-9},
   MRCLASS = {49-01 (28A75 49F20)},
  MRNUMBER = {756417},
MRREVIEWER = {J.\ S.\ Joel},
}

@article {asymptotics,
    AUTHOR = {Simon, Leon},
     TITLE = {Asymptotics for a class of nonlinear evolution equations, with
              applications to geometric problems},
   JOURNAL = {Ann. of Math. (2)},
  FJOURNAL = {Annals of Mathematics. Second Series},
    VOLUME = {118},
      YEAR = {1983},
    NUMBER = {3},
     PAGES = {525--571},
      ISSN = {0003-486X,1939-8980},
   MRCLASS = {58G11 (35B40 49F99 58E20)},
  MRNUMBER = {727703},
MRREVIEWER = {Helmut\ Kaul},
       DOI = {10.2307/2006981},
       URL = {https://doi.org/10.2307/2006981},
}

@article {schocen,
    AUTHOR = {Schoen, Richard M.},
     TITLE = {Uniqueness, symmetry, and embeddedness of minimal surfaces},
   JOURNAL = {J. Differential Geom.},
  FJOURNAL = {Journal of Differential Geometry},
    VOLUME = {18},
      YEAR = {1983},
    NUMBER = {4},
     PAGES = {791--809},
      ISSN = {0022-040X,1945-743X},
   MRCLASS = {53A10 (58E12)},
  MRNUMBER = {730928},
MRREVIEWER = {V.\ M.\ Miklyukov},
       URL = {http://projecteuclid.org/euclid.jdg/1214438183},
}

@book {Hatcher,
    AUTHOR = {Hatcher, Allen},
     TITLE = {Algebraic topology},
 PUBLISHER = {Cambridge University Press, Cambridge},
      YEAR = {2002},
     PAGES = {xii+544},
      ISBN = {0-521-79160-X; 0-521-79540-0},
   MRCLASS = {55-01 (55-00)},
  MRNUMBER = {1867354},
MRREVIEWER = {Donald\ W.\ Kahn},
}

@article {Sogge1,
    AUTHOR = {Sogge, Christopher D.},
     TITLE = {Oscillatory integrals and spherical harmonics},
   JOURNAL = {Duke Math. J.},
  FJOURNAL = {Duke Mathematical Journal},
    VOLUME = {53},
      YEAR = {1986},
    NUMBER = {1},
     PAGES = {43--65},
      ISSN = {0012-7094,1547-7398},
   MRCLASS = {42B10 (42B15)},
  MRNUMBER = {835795},
MRREVIEWER = {Douglas\ Kurtz},
       DOI = {10.1215/S0012-7094-86-05303-2},
       URL = {https://doi.org/10.1215/S0012-7094-86-05303-2},
}

@book {Amiya,
    AUTHOR = {Mukherjee, Amiya},
     TITLE = {Differential topology},
   EDITION = {Second},
 PUBLISHER = {Hindustan Book Agency, New Delhi; Birkh\"auser/Springer, Cham},
      YEAR = {2015},
     PAGES = {xiv+349},
      ISBN = {978-3-319-19044-0; 978-3-319-19045-7},
   MRCLASS = {58-01 (53-01 57-02)},
  MRNUMBER = {3379695},
       DOI = {10.1007/978-3-319-19045-7},
       URL = {https://doi.org/10.1007/978-3-319-19045-7},
}

@article {Chern,
    AUTHOR = {Chern, Shiing-shen and Osserman, Robert},
     TITLE = {Complete minimal surfaces in euclidean {$n$}-space},
   JOURNAL = {J. Analyse Math.},
  FJOURNAL = {Journal d'Analyse Math\'ematique},
    VOLUME = {19},
      YEAR = {1967},
     PAGES = {15--34},
      ISSN = {0021-7670,1565-8538},
   MRCLASS = {53.04},
  MRNUMBER = {226514},
MRREVIEWER = {H.\ B.\ Jenkins},
       DOI = {10.1007/BF02788707},
       URL = {https://doi.org/10.1007/BF02788707},
}

@article {ColdingMincozzi1,
    AUTHOR = {Colding, Tobias H. and Minicozzi, II, William P.},
     TITLE = {The {C}alabi-{Y}au conjectures for embedded surfaces},
   JOURNAL = {Ann. of Math. (2)},
  FJOURNAL = {Annals of Mathematics. Second Series},
    VOLUME = {167},
      YEAR = {2008},
    NUMBER = {1},
     PAGES = {211--243},
      ISSN = {0003-486X,1939-8980},
   MRCLASS = {53A10 (53C42)},
  MRNUMBER = {2373154},
MRREVIEWER = {Fei-Tsen\ Liang},
       DOI = {10.4007/annals.2008.167.211},
       URL = {https://doi.org/10.4007/annals.2008.167.211},
}

@article {collin1,
    AUTHOR = {Collin, Pascal},
     TITLE = {Topologie et courbure des surfaces minimales proprement
              plong\'ees de {$\mathbf R^3$}},
   JOURNAL = {Ann. of Math. (2)},
  FJOURNAL = {Annals of Mathematics. Second Series},
    VOLUME = {145},
      YEAR = {1997},
    NUMBER = {1},
     PAGES = {1--31},
      ISSN = {0003-486X,1939-8980},
   MRCLASS = {53A10},
  MRNUMBER = {1432035},
MRREVIEWER = {Karsten\ Grosse-Brauckmann},
       DOI = {10.2307/2951822},
       URL = {https://doi.org/10.2307/2951822},
}

@article {Meeks2,
    AUTHOR = {Meeks, III, William H. and P\'erez, Joaqu\'in and Ros,
              Antonio},
     TITLE = {Bounds on the topology and index of minimal surfaces},
   JOURNAL = {Acta Math.},
  FJOURNAL = {Acta Mathematica},
    VOLUME = {223},
      YEAR = {2019},
    NUMBER = {1},
     PAGES = {113--149},
      ISSN = {0001-5962,1871-2509},
   MRCLASS = {53A10 (57N05)},
  MRNUMBER = {4018264},
MRREVIEWER = {Min\ Ru},
       DOI = {10.4310/ACTA.2019.v223.n1.a2},
       URL = {https://doi.org/10.4310/ACTA.2019.v223.n1.a2},
}

@article {Choi1,
    AUTHOR = {Choi, Hyeong In and Schoen, Richard},
     TITLE = {The space of minimal embeddings of a surface into a
              three-dimensional manifold of positive {R}icci curvature},
   JOURNAL = {Invent. Math.},
  FJOURNAL = {Inventiones Mathematicae},
    VOLUME = {81},
      YEAR = {1985},
    NUMBER = {3},
     PAGES = {387--394},
      ISSN = {0020-9910,1432-1297},
   MRCLASS = {58E12 (53C42 58D10)},
  MRNUMBER = {807063},
MRREVIEWER = {J.\ Eells},
       DOI = {10.1007/BF01388577},
       URL = {https://doi.org/10.1007/BF01388577},
}

@article {songantoine1,
    AUTHOR = {Song, Antoine},
     TITLE = {Morse index, {B}etti numbers, and singular set of bounded area
              minimal hypersurfaces},
   JOURNAL = {Duke Math. J.},
  FJOURNAL = {Duke Mathematical Journal},
    VOLUME = {172},
      YEAR = {2023},
    NUMBER = {11},
     PAGES = {2073--2147},
      ISSN = {0012-7094,1547-7398},
   MRCLASS = {53A10 (53C42)},
  MRNUMBER = {4627248},
       DOI = {10.1215/00127094-2023-0012},
       URL = {https://doi.org/10.1215/00127094-2023-0012},
}

@book {Hirsch1,
    AUTHOR = {Hirsch, Morris W.},
     TITLE = {Differential topology},
    SERIES = {Graduate Texts in Mathematics},
    VOLUME = {33},
      NOTE = {Corrected reprint of the 1976 original},
 PUBLISHER = {Springer-Verlag, New York},
      YEAR = {1994},
     PAGES = {x+222},
      ISBN = {0-387-90148-5},
   MRCLASS = {57-01 (58-01)},
  MRNUMBER = {1336822},
}

@book{Chodosh1,
    AUTHOR = {Otis Chodosh},
     TITLE = {Introduction to minimal surfaces},
     PUBLISHER = {\url{https://web.stanford.edu/~ochodosh/Math286-min-surf.pdf}},
     YEAR= {2025},
     URL ={https://web.stanford.edu/~ochodosh/Math286-min-surf.pdf}
}

@article {sharp,
    AUTHOR = {Buzano, Reto and Sharp, Ben},
     TITLE = {Qualitative and quantitative estimates for minimal
              hypersurfaces with bounded index and area},
   JOURNAL = {Trans. Amer. Math. Soc.},
  FJOURNAL = {Transactions of the American Mathematical Society},
    VOLUME = {370},
      YEAR = {2018},
    NUMBER = {6},
     PAGES = {4373--4399},
      ISSN = {0002-9947,1088-6850},
   MRCLASS = {53A10 (49Q05 58E12)},
  MRNUMBER = {3811532},
MRREVIEWER = {Peter\ McGrath},
       DOI = {10.1090/tran/7168},
       URL = {https://doi.org/10.1090/tran/7168},
}

@article {edelen,
    AUTHOR = {Edelen, Nick},
     TITLE = {Degeneration of 7-dimensional minimal hypersurfaces which are
              stable or have a bounded index},
   JOURNAL = {Arch. Ration. Mech. Anal.},
  FJOURNAL = {Archive for Rational Mechanics and Analysis},
    VOLUME = {248},
      YEAR = {2024},
    NUMBER = {4},
     PAGES = {Paper No. 65, 73},
      ISSN = {0003-9527,1432-0673},
   MRCLASS = {58E12 (49Q05 53A10)},
  MRNUMBER = {4768490},
MRREVIEWER = {Panayotis\ Vyridis},
       DOI = {10.1007/s00205-024-02003-w},
       URL = {https://doi.org/10.1007/s00205-024-02003-w},
}

\end{document}